\documentclass[reqno]{amsart}
\usepackage{amssymb,amscd,verbatim, amsthm,graphicx, color}
\usepackage[enableskew,vcentermath]{youngtab}
\usepackage{srcltx}
\usepackage[mathscr]{eucal}
\usepackage[usenames,dvipsnames]{xcolor}

\usepackage{longtable}

\newcommand \C[1]{{\mathcal #1}}
\newcommand \ovl[1]{{\overline {#1}}}

\newcommand \bb[1]{{\mathbb #1}}

\newcommand \wti[1]{{\widetilde {#1}}}

\newcommand\fg{\mathfrak g}

\newcommand \bA{{\mathbb A}}
\newcommand \bC{{\mathbb C}}
\newcommand \bF{{\mathbb F}}
\newcommand \bH{{\mathbb H}}
\newcommand \bR{{\mathbb R}}
\newcommand \bZ{{\mathbb Z}}

\newcommand\one{1\!\!1}

\newcommand\CB{{\C B}}
\newcommand\CG{{\C G}}
\newcommand\CH{{\C H}}

\newcommand\CO{{\C O}}

\newcommand\ie{{\it i.e.~}}

\newcommand\ep{{\epsilon}}
\newcommand\la{{\lambda}}
\newcommand\om{{\omega}}
\newcommand\al{{\alpha}}
\newcommand\sig{{\sigma}}

\newcommand\sgn{\mathsf{sgn}}
\newcommand\triv{\mathsf{triv}}

\newtheorem{proposition}{Proposition}[subsection]
\newtheorem{corollary}[proposition]{Corollary}
\newtheorem{lemma}[proposition]{Lemma}
\newtheorem{theorem}[proposition]{Theorem}
\newtheorem{conjecture}[proposition]{Conjecture}

\theoremstyle{definition}
\newtheorem{definition}[proposition]{Definition}
\newtheorem{remark}[proposition]{Remark}

\newtheorem{example}[proposition]{Example}

\newcommand{\clrblu}{\color{blue}}

\newcommand{\clrr}{\color{red}}

\newcommand\Hom{\operatorname{Hom}}
\newcommand\Ind{\operatorname{Ind}}

\newcommand\End{\operatorname{End}}

\newcommand\Id{\operatorname{Id}}
\newcommand\Ad{\operatorname{Ad}}

\newcommand\im{\operatorname{Im}}

\newcommand\cc{\textsf{cc}}

\newcommand\St{\textsf{St}}
\newcommand\reg{\textsf{reg}}

\newcommand{\lam}{\lambda}
\newcommand{\ol}{\overline}

\usepackage[mathscr]{eucal}

\newcommand{\scH}{\mathscr{H}}

\newcommand{\scP}{\mathscr{P}}
\newcommand{\scR}{\mathscr{R}}
\newcommand{\scT}{\mathscr{T}}

\newcommand\deltaunit{\delta_{\Ca I^-B,\delta_{A_0,\one}}}

\numberwithin{equation}{subsection}

\newfont{\lge}{cmmi10 scaled 1640}

\newcommand \Ca[1]{{\mathcal #1}}

\newcommand{\La}{\Lambda}

\newcommand \ev{\mathsf{ev}}
\newcommand\tr{\operatorname{tr}}

\begin{document}



\bigskip
\title{Star operations for affine Hecke algebras} 
\author{Dan Barbasch}
       \address[D. Barbasch]{Dept. of Mathematics\\
               Cornell University\\Ithaca, NY 14850}
       \email{barbasch@math.cornell.edu}

\author{Dan Ciubotaru}
        \address[D. Ciubotaru]{Mathematical Institute\\ University of
          Oxford\\ Andrew Wiles Building\\ Oxford, OX2 6GG, UK}
        \email{dan.ciubotaru@maths.ox.ac.uk}


\begin{abstract}In this paper, we consider the star operations for
  {(graded) affine Hecke algebras} which preserve certain natural
  filtrations. We show that, up to inner conjugation, there are only two
  such star operations for the graded Hecke algebra: the first, denoted $\star$,  corresponds to
  the usual star operation from reductive $p$-adic groups, and the second,
  denoted $\bullet$ can be regarded as the analogue of the compact
  star operation of a real group considered by \cite{ALTV}. We explain how the star operation $\bullet$ appears naturally in the Iwahori-spherical setting of $p$-adic groups via the endomorphism algebras of Bernstein projectives. We also
  prove certain results about the signature of $\bullet$-invariant
  forms and, in particular, about $\bullet$-unitary simple modules.
\end{abstract}

\dedicatory{{To Wilfried Schmid with admiration}}
\maketitle

\bigskip

\bigskip
\setcounter{tocdepth}{1}
\tableofcontents


\section{Introduction}\label{sec:1}
{This work is motivated by the results about the unitary dual obtained 
in the case of real reductive groups by Adams, van Leeuwen, Trapa, and Vogan \cite{ALTV} on the one hand, and on the other 
hand, in work and conjectures of Schmid and Vilonen \cite{SV}. The ultimate goal is to 
obtain an algorithm for computing hermitian forms of irreducible modules in 
the case of reductive $p$-adic groups}.

\smallskip

In this paper, {we initiate the study of invariant hermitian forms for 
the (graded) affine Hecke algebras} that appear in the theory of 
{unipotent representations of} reductive $p$-adic groups (Lusztig \cite{L4}). 
There are two main parts to our paper which we explain next.

\subsection{} We introduce  star operations (conjugate-linear
involutive anti-automorphims) for the affine
  Hecke algebra $\C H$ with unequal parameters which preserve natural
  filtrations of $\C H$ (section \ref{sec:2}) and classify them in the corresponding setting of the graded affine Hecke algebras defined by Lusztig \cite{L4}. The classification problem can be viewed as an
  analogue of the 
  problem of classifying the star operations for the enveloping
  algebra $U(\fg)$ of a complex semisimple Lie algebra which preserve
  $\fg.$   Proposition \ref{p:classification} 
  says that essentially there are only two such star operations: $\star$
  and $\bullet$, see Definitions \ref{d:basicstars-affine} and \ref{d:basicstars}. 

The anti-automorphism
  $\star$ is known to correspond to the natural star operation of the
  Hecke algebra of a reductive $p$-adic group, i.e., $f^\star(g)=\overline
  {f(g^{-1})}$, see \cite{BM1,BM2}. 

On the other hand, the anti-automorphism $\bullet$ is the Hecke
algebra analogue of the ``compact star operation'' for $(\fg,
K)$-module studied in \cite{ALTV}. In section \ref{sec:aff}, we explain that
$\bullet$ appears naturally in the study of Iwahori-Hecke algebras via
the projective (non-admissible) modules defined by Bernstein \cite{Ber}. The operation
$\bullet$ for affine Hecke algebras also arises naturally in work of Opdam \cite{O2}.

\subsection{} We study the basic properties of the signature of
$\bullet$-invariant hermitian forms for finite dimensional
$\bH$-modules. We
explain that every irreducible spherical $\bH$-module with real central character
admits an (explicit) nondegenerate $\bullet$-invariant hermitian form,
Proposition \ref{p:bullet-form-minimal}. This result is generalized in \cite{BC4}, where we prove that every simple
$\bH$-module with real central character admits a nondegenerate
$\bullet$-invariant hermitian form, and, moreover, this form can be normalized
canonically (at least when $\bH$ is of geometric type in the sense of
Lusztig) to be positive definite on the lowest $W$-types.
These results can be thought of Hecke algebra analogue of the similar
results about c-invariant forms of $(\fg,K)$-modules
\cite{ALTV}.  The formulations of some of our results were inspired by
the ongoing work of Schmid and Vilonen \cite{SV} aimed at using geometric
methods to study unitarity. For example, the relation between the $\star$- and
$\bullet$-signature characters of the   tempered modules with real
central characters (realized in the  cohomology of Springer fibers) in
section \ref{sec:3.2} is motivated by their work. 

Using the Dirac operator defined in \cite{BCT}, we also prove that the
only $\bullet$-unitary spherical $\bH$-modules, where $\bH$ has equal
parameters, are the ones whose 
parameters lie in the closure of the $\rho^\vee$-cone, Theorem
\ref{t:spherical}. Similar ideas lead to showing 
that every simple $\bH$-module with central character
$\rho^\vee$ has nontrivial Dirac cohomology, and moreover their Dirac
cohomology spaces are essentially the same, Corollary \ref{c:Dirac-coh-rho}.

{\subsection{} {These results were presented in two talks given by Dan Barbasch in 2013}. 
The first was at the TSIMF conference at
Sanya in January, as part of a special session organized by
W. Schmid and S. Miller, the second at the conference in honor of W. Schmid's 
$70^{th}$ birthday in June of 2013. He would like to thank
the organizers of these conferences for providing  the means for
mathematical researchers to have a very productive exhange of ideas.  

\smallskip

The first author was partially supported by NSF grants DMS-0967386, DMS-0901104
and an NSA-AMS grant. The second author was partially supported by NSF DMS-1302122 and NSA-AMS 111016.  }

\section{Star operations: the affine Hecke algebra}\label{sec:aff}

\subsection{The affine Hecke algebra} Let $\C R=(X,R,X^\vee, R^\vee,\Pi)$ be a based root datum \cite{Sp}. In
particular, $X,X^\vee$ are lattices in perfect duality
$\langle~,~\rangle:X\times X^\vee\to\bZ$, $R\subset X\setminus\{0\}$
and $R^\vee\subset X^\vee\setminus\{0\}$ are the (finite) sets of
roots and coroots respectively, and $\Pi\subset R$ is a basis of
simple roots. Let $W$ be the finite Weyl group with set of generators
$S=\{s_\al:\al\in \Pi\}.$ Set $W^e=W\ltimes
X$, the extended affine Weyl group, and $W^a=W\ltimes Q$, the affine
Weyl group, where $Q$ is the root lattice of $R$. 

The set $R^a=R^\vee\times \bZ\subset X^\vee\times\bZ$ is the set of
affine roots. A basis of simple affine roots is given by
$\Pi^a=(\Pi^\vee\times\{0\})\cup\{(\gamma^\vee,1): \gamma^\vee\in R^\vee
\text{ minimal}\}.$ For every affine root $\mathbf a=(\al^\vee,n)$, let
$s_{\mathbf a}:X\to X$ denote the reflection
$s_{\mathbf a}(x)=x-((x,\al^\vee)+n)\al.$ The affine Weyl group $W^a$
has a set of generators $S^a=\{s_{\mathbf a}: \mathbf a\in \Pi^a\}$.
Let $\ell:W^e\to\bZ$ be the length function with respect to $S^a$.

Let $\mathbf q$ be an indeterminate and let 
$$L: S^a\to \bZ_{\ge 0}$$
be a $W^a$-invariant function.

\begin{definition}[Iwahori presentation] The affine
  Hecke algebra $\C H(\mathbf q)=\C H(\C R,\mathbf q,L)$ associated to the root datum
  $\C R$ and parameters $L$ is the unique associative,
  unital $\bC[\mathbf q,\mathbf q^{-1}]$-algebra with basis $\{T_w: w\in W^e\}$ and
  relations
\begin{enumerate}
\item[(i)] $T_w T_{w'}=T_{ww'},$ for all $w,w'\in W^e$ such that
  $\ell(ww')=\ell(w)+\ell(w')$;
\item[(ii)] $(T_s-\mathbf q^{2L(s)})(T_s+1)=0$ for all $s\in S^a.$
\end{enumerate}
\end{definition}

\subsection{The Bernstein presentation} The affine Hecke algebra admits a second presentation due to Bernstein and Lusztig, see \cite{L1}.

A parameter set for $\C R$ is a pair of functions $(\lambda,\lambda^*)$,
$$
\begin{aligned}
 &\lambda:\Pi\to \bZ_{\ge 0},&\lambda^*:\{\al\in\Pi:\al^\vee\in 2X^\vee\}\to \bZ_{\ge 0}, 
\end{aligned}
$$
such that $\lambda(\al)=\lambda(\al')$ {and
  $\la^*(\al)=\la^*(\al')$} whenever $\al,\al'$ are $W-$conjugate. 
The relation with the parameters in the Iwahori presentation is:
\begin{equation}
\lambda(\al)=L(s_\al),\ \al\in\Pi,\quad \lambda^*(\al)=L(\hat s_\al), \ \al\in\Pi, \al^\vee\in 2 X^\vee,
\end{equation}
where $\hat{\ }$ is the unique nontrivial automorphism of the Dynkin diagram of affine type $\wti C_r$.

\begin{definition}[Bernstein presentation]\label{d:2.1} The affine Hecke algebra
  $\CH(\mathbf q)=\CH^{\lambda,\lambda^*}(\C R,\mathbf q)$ associated to the root datum
  $\C R$ with parameter set $(\lambda,\lambda^*),$ is the
  associative algebra over $\bC[\mathbf q,\mathbf q^{-1}]$ with unit, defined by generators $T_w$, $w\in W$, and $\theta_x$,
  $x\in X$ with relations:
\begin{align}
&(T_{s_\al}+1)(T_{s_\al}-\mathbf q^{2\lambda(\al)})=0,\text{ for all }\al\in\Pi,\\
&T_wT_{w'}=T_{ww'},\text{ for all }w,w'\in W\text{ such that }\ell(ww')=\ell(w)+\ell(w'),\notag\\
&\theta_x\theta_{x'}=\theta_{x+x'},\text{ for all }x,x'\in X,\\
&\theta_x T_{s_\al}-T_{s_\al}\theta_{s_\al(x)}=(\theta_x-\theta_{s_\al(x)}) \C
  (\C G(\al)-1),\text{ where } x\in X, \al\in\Pi,\text{ and } \notag\\
&\C G(\al)=\begin{cases} 
\frac{\theta_\al \mathbf q^{2\lambda(\al)}-1}{\theta_\al-1}, &\text{ if }
\al^\vee\notin 2X^\vee,\\
    \frac{(\theta_\al \mathbf q^{\lambda(\al)+\lambda^*(\al)}-1)(\theta_\al
      \mathbf q^{\lambda(\al)-\lambda^*(\al)}+1)}{\theta_{2\al}-1}, &\text{ if
    }\al^\vee\in 2X^\vee.
 \end{cases}
\end{align}
\end{definition}

We refer to \cite[section 3]{L1} for more details about the relations between the two presentations. 
 
\subsection{Star operations} There are two known star operations on $\C H(\mathbf q)$, i.e., conjugate linear involutive anti-automorphisms, $\star$ and $\bullet$.

\begin{definition}\label{d:basicstars-affine}
\begin{enumerate}
\item In the Iwahori presentation, $\star$ is defined on generators by
\begin{equation}
\mathbf q^\star=\mathbf q,\quad T_w^{\star}=T_{w^{-1}},\ w\in W^e.
\end{equation}
In the Bernstein presentation, the equivalent definition is \cite[section 5]{BM2}:
\begin{equation}
\mathbf q^\star=\mathbf q,\quad T_w^\star=T_{w^{-1}},\ w\in W,\quad \theta_x^\star=T_{w_0}\cdot \theta_{-w_0(x)}\cdot T_{w_0}^{-1},\ x\in X,
\end{equation}
where $w_0$ is the long Weyl group element of $W$.
\end{enumerate}
\item The $\bullet$ operation is defined in the Bernstein presentation
\begin{equation}
\mathbf q^\bullet=\mathbf q,\quad T_w^\bullet=T_{w^{-1}},\ w\in W,\quad \theta_x^\bullet=\theta_x,\ x\in X.
\end{equation}
The fact that this definition extends to a star operation is equivalent with the fact that the algebra $\C H(\mathbf q)$ is isomorphic to its opposite algebra.
\end{definition}
{
\subsection{} The central characters of $\CH$ are 
parametrized by  
$W-$orbits in the torus $\C T:=X\otimes_{\bb Z} \bb C^*.$ 
In \cite{BM2} (following \cite{L1}), 
a filtration of $\C H$ is defined for any finite $W-$invariant 
set $\C O\subset\C T.$ Let $\C O_a$ be the $W$-orbit of an element 
$a\in \C T.$ The $\theta_x$ are interpreted as regular functions 
$R(\C T)$ on $\C T.$ The filtration associated to $\C O_a$ is defined by the powers of 
the ideal $\C I_a:=R(\C O_a)\C H$ generated by 
\begin{equation}
  \label{eq:ideal}
R(\C O_a):=\{f\in R(\C T\times\bb C^\times)\ :\ f(\sig,1)=0\ \text{ for any } 
\sig\in \C O_a\}.
\end{equation}
The graded algebra $\bb H_a$ is then shown, \cite[Proposition 4.4]{L1} and \cite[Proposition 3.2]{BM2} to be a matrix algebra over an 
appropriate graded affine Hecke algebra as in Definition \ref{d:graded}.

Let $\kappa$ be an automorphism (or anti-automorphism) of $\C H$. Then $\kappa$ induces an automorphism of the center $Z(\C H)$, and therefore an isomorphism $$\hat\kappa: \C T\times \bC^\times \to \C T\times \bC^\times.$$ 
We will only consider morphisms $\kappa$ that fix $\mathbf q$, and thus $\hat\kappa$ restricts to an isomorphism of $\C T$ as well.

\begin{definition}
An automorphism (or anti-automorphism) $\kappa$ of $\C H$ is called admissible, if $\kappa(\mathbf q)=\mathbf q$, $\kappa(T_w)=T_w,$ for all $w\in W$,  and 
$\kappa(\C I_a)\subset \C I_{\hat\kappa(a)}$ for all $a\in \C T$. It is clear that the operations $\bullet$ and $\star$ from Definition \ref{d:basicstars-affine} are admissible in this sense.
\end{definition}

In section \ref{sec:2}, we study the analogues of admissible automorphisms and anti-automorphisms for graded affine Hecke algebras. Motivated by the main result of that section, Proposition \ref{p:classification} and the connection between the affine Hecke algebra $\C H$ and the graded Hecke algebras $\bH_a$, we make the following conjecture.

\begin{conjecture}\label{c:class}
Let $\kappa$ be an admissible involutive anti-automorphism.
\begin{enumerate}
\item If $\hat\kappa(a)=a$, for all $a\in \C T$, then $\kappa(\theta_x)=\theta_x^\bullet$, for all $x\in X$. 
\item If $\hat\kappa(a)=a^{-1}$, for all $a\in \C T$, then $\kappa(\theta_x)=\theta_x^\star$, for all $x\in X$.
\end{enumerate}
\end{conjecture}

} 

\subsection{}The operation $\star$ appears naturally in relation with smooth representations of reductive $p$-adic groups. Suppose $\bF$ is a $p$-adic field of characteristic $0$ with residue field $\bF_q$. Let $\CG$ be the group of $\bF$-rational points of a connected reductive algebraic group defined over $\bF.$ Let $I$ be an Iwahori subgroup of $\C G$ and let $\C C_I(\CG)$ be the category of smooth admissible $\C G$-representations which are generated by their $I$-fixed vectors. 

On the other hand, let $\CH(\CG,I)$ be the Iwahori-Hecke algebra, i.e., the convolution algebra (with respect to a Haar measure of $\C G$) of compactly-supported complex valued functions on $G$ that are $I$-biinvariant. By a classical result of Borel, the functor $V\mapsto V^I$ induces an equivalence of categories between $\C C_I(\CG)$ and the category of finite dimensional $\CH(\CG,I)$-modules.

The Iwahori-Hecke algebra $\CH(\CG,I)$ is a specialization of $\C H(\mathbf q)$ with $\mathbf q=\sqrt q$ and the appropriate specialization of parameters $L$ , see \cite{Ti}. Under this specialization, the natural star operation $$f^\star(g)=\overline {f(g^{-1})}, \quad f\in \CH(\CG,I),$$
on $\C H(\CG,I)$ corresponds to the operation $\star$ on $\C H(\mathbf q)$.

\subsection{Bernstein's projective modules}\label{sec:Bernstein}
In the rest of this section, we explain how the $\bullet$-form for affine Hecke
algebras appears naturally when the Iwahori-Hecke algebras are viewed
as endomorphism algebras of the Bernstein projective modules
\cite{Ber}, see also \cite{He}.

Let $V$ be a complex vector space, 
$$
V^h:=\bigg\{\lam:V\longrightarrow \bC\
:\ \lam (\al_1v_1 + \al_2 v_2)=\ol\al_1\lam(v_1)+\ol\al_2\lam(v_2)\bigg\}.
$$ 
A sesquilinear form is a bilinear form $\langle\cdot,\cdot\rangle$ which
is linear in the first variable, conjugate linear in the second
variable. This is the same as a complex linear map
$\lam:V\longrightarrow V^h.$ The relation is
\[
\langle v,w\rangle_\lam=\lam(v)(w).
\]
Such a form is called nondegenerate if $\lam$ is injective. 
To any sesquilinear form $\lam$ there is associated $\lam^h:V\subset
(V^h)^h\longrightarrow V^h,$ $\lam^h(v)(w):=\ol{\lam(w)(v)}.$ The form is
  called symmetric, if $\lam=\lam^h.$ A symmetric form is an inner
  product if $\lam(v)(v)\ge 0,$ with equality if and only if $v=0.$

\

{Let $G$ be a reductive $p$-adic group. If $(\pi,V)$ is a representation of $G,$ then $(\pi^h,V^h)$ is the representation defined as 
$$
\big(\pi^h(g)\la\big)(v):=\la(\pi(g^{-1})v).
$$
}

\subsection{The projective $\C P$} Let $M$ be a Levi subgroup of $G$. Denote by $M_0$ the intersection of
the kernels of all the unramified characters of $M$. Let $\wti\sigma$ be a relative supercuspidal representation of $M$, $\sig_0$ a supercuspidal constituent of $\wti\sig\mid_{M_0}$.

Define
\[
\begin{aligned}
&\big(\sig,\ V_{\sig}=\Ind_{M_0}^M\sig_0\big)_c &&\text{ induction with compact support, }\\
&\big(\Pi,\ \scP=\Ind_P^G V_\sig\big) &&\text{ normalized induction. }
\end{aligned}
\]
A typical element of $\sig$ is $\delta_{mM_0,v}$ with $m\in M/M_0$ and $v\in V_{\sig_0}.$ This is the \textit{delta}-function supported on the coset $mM_0$ taking constant value $v.$

A typical element of $\scP$ is given by $\delta_{UxP,\delta_{mM_0,v}}$
where $U\in G/P$ is a neighborhood of the identity, the function
satisfies the appropriate transformation law under $P$ on the right,
and the value at $x$ is $\delta_{mM_0,v}.$

\smallskip

If $\psi\in\Hom_G[\scP,\scP],$ then $\psi^h\in\Hom_G[\scP^h,\scP^h].$
But $\scP$ admits a $G-$invariant positive definite hermitian form, so
while $\scP\ne \scP^h,$ nevertheless there is an inclusion $\iota:\scP\longrightarrow \scP^h.$ 
More precisely, if $\scP=\Ind_P^G\sig,$ then the hermitian dual $\scP^h$ is  naturally
isomorphic to $\Ind_P^G\sig^h.$ If $\lam:G\longrightarrow V_\sig^h$ is
such that $\lam(xp)=\sig^h(p^{-1})\lam (g),$ and $f:G\longrightarrow V$
is such that $f(gp)=\sig(p^{-1})f(x),$ then the pairing is
\[
\langle\la, f\rangle:=\int_{G/P} \lam(x)(f(x))dx.
\]
When $\sig$ is unitary (or just has a nondegenerate form so that
 $\sig\subset\sig^h$), we get $\scP\subset\scP^h$ via 
\[
g\in\scP\mapsto \la_g\in \scP^h,\quad 
\lam_g(f)=\int_{G/P}\langle f(x),g(x)\rangle dx \text{ for } f\in\scP.
\]

\subsection{Inner Product} We recall two classical results.

\begin{theorem}[Frobenius reciprocity, {\cite[Theorem 3.2.4]{Cas}}]
\[
\Hom_G[V,\scP]\cong\Hom_{M}[V_N,\sig\delta_P^{-1}].  
\]
\end{theorem}
\begin{theorem}[Second adjointness, {\cite[Theorem 20]{Ber}}]
\[
\Hom_G[\scP,V]\cong\Hom_M[\delta_{\ol P}^{-1}\sig,V_{\ol N}]\cong\Hom_{M_0}[\sig_0,\delta_P^{-1}V_{\ol N}].
\]  
\end{theorem}

\bigskip
Let $\ol\scP$ be the module induced from $\sig$ from the opposite
parabolic $\ol P:=M\ol N.$ The (second) adjointness theorem gives
\[
\begin{aligned}
&\Hom_G[\scP,\scP]=\Hom_M[\delta_P^{-1}\sig,\scP_{\ol N}]=
\Hom_{M_0}[\sig_0,\delta_P^{-1}\scP_{\ol N}],\\
&\Hom_G[\ol\scP,\scP]=\Hom_M[\delta_{\ol P}^{-1}\sig,\scP_N]=
\Hom_{M_0}[\sig_0,\delta_P\scP_N].
\end{aligned}
\]
Assume $P$ and $\ol P$ are conjugate, and let $w_0\in W$ be the shortest Weyl group element taking $P$ to $\ol P$, stabilizing $M$ and taking $N$ to $\ol N$. Assume also that there is an isomorphism $\tau_0:(\sig_0,V_{\sig_0})\longrightarrow (w_0\circ\sig_0,V_{\sig_0})$. Extend it to $\tau:V_\sig\longrightarrow V_\sig$ by  
$\tau(\delta_{mM_0,v})=\delta_{w_0(m)M_0,\tau(v)}.$ Write $\wti\tau$ for the isomorphism
\begin{align*}
&\wti\tau:\ol\scP\longrightarrow\scP,\\
&\wti\tau(f)(x)=\tau(f(xw_0)).  
\end{align*}
Thus given $\Phi,\Psi\in\Hom_G[\scP,\scP],$ then $\ol\Phi:=\Phi\circ\wti\tau\in \Hom_G[\ol\scP,\scP]$, and they give rise to 
\[
\begin{aligned}
&\ol\phi\in\Hom_{M_0}[\sig_0,\scP_{\ol N}]\\
&\psi\in\Hom_{M_0}[\sig_0,\scP_N].
\end{aligned}
\]
According to Casselmann \cite[Proposition 4.2.3]{Cas}, there is a nondegenerate pairing $\langle\ ,\ \rangle_{N,\ol N}$ between $\scP_N$ and $\scP_{\ol N}$.  Given $v_1,v_2\in V_{\sig_0}$, we can form
\[
\langle
v_1,v_2\rangle_{\Phi,\Psi}:=\langle\ol\phi(v_1),\psi(v_2)\rangle_{N,\ol
  N}.
\]
This pairing is invariant and sesquilinear, so there is a constant $m_{\Phi,\Psi}$ such that
\begin{equation}\label{m-Phi-Psi}
\langle v_1,v_2\rangle_{\Phi,\Psi}=m_{\Phi,\Psi}\langle v_1,v_2\rangle_{\sig_0}.
\end{equation}
We define a sesquilinear pairing
\begin{equation}\label{form}
\langle\Phi,\Psi\rangle:=m_{\Phi,\Psi}.
\end{equation}
\subsection{} We make the form (\ref{form}) precise. Let $K_\ell$ be an open compact subgroup with an Iwasawa decomposition compatible with $P,$ \ie $K_\ell=K_\ell^-\cdot K_\ell^0\cdot K_\ell^+,$ invariant by $w_0.$ 

Let $x_0,y_0\in V_{\sig_0}^{K_\ell^0},$ and $x:=\delta_{M_0,x_0}, y:=\delta_{M_0,y_0}.$  Then $\delta_{K_\ell^+\ol P,x}\in\ol\scP$ and $\delta_{K_\ell^-P,y}\in\scP.$ The isomorphism $\wti\tau$ takes $\delta_{K_\ell^+\ol P,x}$ to $\delta_{K_\ell^+w_0P,\tau(x)}.$ So
\[
(x_0,y_0)_{\Phi,\Psi}:=\langle\ol\Phi_{\ol N}\big(\delta_{K_\ell^+w_0P,\tau(x)}\big),\Psi_N\big(\delta_{K_\ell^-P,y}\big)\rangle_{\ol N,N}=m_{\Phi,\Psi}\langle x_0,y_0\rangle_{\sig_0}.
\]
Here $\ol\Phi_{\ol N}$ and $\Psi_N$ are the projection maps onto $\scP_{\ol N}$ and $\scP_N$ respectively.

Let $\La\in A:=Z(M)$ be such that it is regular on $N$ and \textit{contracts it}. Let $a(\La)$ and $a(-\La)$ be the $K_\ell$ double cosets of $\La$ and its inverse.  

By Casselman \cite[section 4]{Cas} and Bernstein \cite[chapter III.3]{Ber},

\[
\begin{aligned}
&\scP^{a(-\La),K_\ell}\cong \scP_{\ol N}^{K_\ell^0},\\
&\scP^{a(\La),K_\ell}\cong \scP_{N}^{K_\ell^0},
\end{aligned}
\]  
because $a(\La)$ contracts $K_\ell^+.$ We conclude that
$$
\begin{aligned}
&\delta_{K_\ell P,x}\in\scP^{a(-\La),K_\ell},\quad &&\text{ so }\quad
\Phi(\delta_{K_\ell P,x})\in\scP^{a(-\La),K_\ell}\cong\scP_{\ol N}^{K_\ell^0},\\
 &\delta_{K_\ell w_0P,\tau_0y}\in\scP^{a(\La),K_\ell},\quad &&\text{  so }\quad 
\Psi(\delta_{K_\ell w_0 P,\tau_0 y})\in\scP^{a(\La),K_\ell}\cong \scP_N^{K_\ell^0}.
\end{aligned}
$$

\begin{proposition}
With the notation as in (\ref{m-Phi-Psi}), $
m_{\Phi,\Psi}=\ol{m_{\Psi,\Phi}}.
$  {In other words, the sesquilinear form (\ref{form}) is hermitian}.
\end{proposition}
\begin{proof}
Assume $\tau_0\ne -Id,$ or else use $-\tau_0.$ Thus there is $x_0$
such that $\tau_0x_0=x_0.$ Let $f_{w_0}:=\delta_{K_lw_0K_l}.$ Then
$f^*_{w_0}=f_{w_0},$ and 
\begin{align*}
&\Pi(f_{w_0})\delta_{K_lw_0 P,x}=\delta_{K_\ell P,x},\\
&\Pi(f_{w_0})\delta_{K_\ell P,x}=\delta_{K_\ell w_0 P,x}.
\end{align*}
Then
\begin{align*}
m_{\Phi,\Psi}<x_0,x_0>&=
<\Phi(\delta_{K_\ell P,x}),\Psi(\delta_{K_\ell w_0P,x})>=\\
&=<\Phi(\Pi(f_{w_0})\delta_{K_\ell w_0 P,x}),\Psi(\delta_{K_\ell w_0P,x})>=\\
&= <\Phi(\delta_{K_\ell w_0P,x}),\Psi(\Pi(f_{w_0})\delta_{K_\ell w_0P,x})>=\\
&=<\Phi(\delta_{K_\ell w_0P,x}),\Psi(\delta_{K_\ell P,x})>=\\
&=\ol{<\Psi(\delta_{K_\ell P,x}),\Phi(\delta_{K_\ell w_0P,x})>}= \\
&=\ol{m_{\Psi,\Phi}}<x_0,x_0>. 
\end{align*}
\end{proof}
\subsection{}
For $a\in A,$ let $\Theta_a\in \Hom_G[\scP,\scP]$ be given by
{
\begin{equation}
    \label{deftheta}
    \Theta_a(\delta_{K_\ell gP,x})=\delta_{K_\ell gP,\theta_a(x)},\qquad 
\theta_a(x):=\theta_a(\delta_{mM_0,x_0})=\delta_{maM_0,x_0}.
\end{equation}
}
\begin{proposition}  \label{p:starfortheta}
\begin{equation*}
  <\Phi,\Psi\circ\Theta_a>=<\Phi\circ\Theta_a,\Psi>.
\end{equation*}
\end{proposition}
  \begin{proof}
There is $f_a\in \scH(K_\ell\backslash G/K_\ell)$
(namely $\delta_{K_\ell aK_\ell}$) such that $\Theta_a(\delta_{K_\ell
  P,x})=\Pi(f_a)(\delta_{K_\ell   P,x}).$ 
{Then  use the fact that $f_a^*=f_{a^{-1}}$ for $a\in A^+$ \textit{dominant}}.  
\end{proof}

\subsection{Digression about the intertwining  operator}
Let $J:\scP\longrightarrow\scP$ be given by  the formula
\begin{equation}
  \label{eq:intertwiningoperator}
  Jf(x):=\int_N\tau_0f(xnw_0)\;dn=\int_{\ol N}\tau_0f(xw_0\ol n)\;d\ol n.
\end{equation}
{This should be considered as a formal expression. When you specialize
  to a value $\nu\in \widehat A,$ the split part of the center of $M,$
  $J$ will have poles. }

Recall the inner product on $\scP$,
\[
\langle f_1,f_2\rangle:=\int_{K_0}\langle f_1(k),f_2(k)\rangle\; dk.
\] 
\begin{proposition}\label{p:adjointness}
  $$
<Jf_1,f_2>=<f_1,Jf_2>.
$$
\end{proposition}
\begin{proof}
  \begin{equation}
    \label{eq:6.2.1}
    <f_1,Jf_2>=\int_{K_0}<f_1(k),\int_{\ol N}\tau_0f_2(kw_0\ol n)\;
    d\ol{n}>\; dk.
  \end{equation}
We can move $w_0$ and $\tau_0$ to the other side:
\begin{equation}
  \label{eq:6.2.2}
<f_1,Jf_2>=\int_{K_0}<\tau_0f_1(kw_0),\int_{\ol N}f_2(k\ol n)\; d\ol
n>\; dk.  
\end{equation}
Write $\ol{n}=\kappa(\ol n)\cdot n(\ol n)\cdot m(\ol n).$ So
\begin{equation}
\label{eq:6.2.3}
\begin{aligned}
&<f_1,Jf_2>= 
\int_{K_0}<\tau_0f_1(kw_0),
\int_{\ol N}\sig(m(\ol n)^{-1})f_2(k\kappa(\ol n))\; d\ol n>\; dk=
\\
&=\int_{K_0}<\int_{\ol N}\sig(m(\ol n))\tau_0f_1(k\kappa(\ol
n)^{-1}w_0)\; d\ol n,f_2(k)>\; dk. 
\end{aligned}  
\end{equation}
Since $\kappa(\ol n)=\ol n\cdot m(\ol n)^{-1}\cdot n(\ol n)^{-1},$ we conclude
$\kappa(\ol n)^{-1}=n(\ol n)\cdot m(\ol n)\cdot \ol n^{-1}.$ So
\begin{equation}
    \label{eq:6.2.4}
\begin{aligned}
<f_1,Jf_2>&=\\
=\int_{K_0}&<\int_{\ol N}\sig(m(\ol n))\tau_0f_1(k n(\ol n) m(\ol
n)\ol n^{-1}w_0)\; d\ol n,f_2(k)>\;dk= \\
=\int_{K_0}&<\int_{\ol N}
\tau_0f_1(kn(\ol n))\; d\ol n,f_2(k)>\;dk
\end{aligned}
\end{equation}
because $\sig(m(\ol n))$ is conjugated by $w_0$ but then flipped
back by $\tau_0,$ and then cancels $\sig(m(\ol n)).$ 
Finally
\[
Jf_1=\int_{\ol N}\tau_0f_1(kn(\ol n)w_0)\; d\ol n
\]
follows from the fact that $\ol
n\mapsto w_0n(\ol n)w_0^{-1}$ is an isomorphism with trivial Jacobian.
\end{proof}
\subsection{} Assume from now on that $G$ is a split $p$-adic
group. Let $P=B=AN$ be a Borel subgroup.
Let $K_0$ be the hyperspecial maximal compact subgroup, and
$K_1\subset \Ca I\subset K_0$ be an Iwahori subgroup.
It has an Iwasawa decomposition $\Ca I=\Ca I^-\cdot A_0\cdot\Ca I^+.$ Furthermore, 
$G=KB=\cup \Ca IwB$ disjoint union where $w\in W.$

We consider the case of the trivial representation of $A_0:=K_0\cap
A$, $\sig_0=\triv$, i.e, this is the case of representations with $\Ca
I-$fixed vectors. Let $\CH=\CH(\C I\backslash G/\C I)$ be the
Iwahori-Hecke algebra of compactly supported smooth $\C I$-biinvariant
functions with convolution with respect to a Haar measure.

\begin{proposition}
In the Iwahori-spherical case, the algebra $\Hom[\scP,\scP]$ is naturally isomorphic to the opposite
algebra to $\Ca H(\C I\backslash G/\C I).$
\end{proposition}

\begin{proof}
Recall 
$$
\Hom_G[\scP,\scP]\cong \Hom_{A_0}[\sig_0,\scP_{\ol N}]\cong\scP^{A_0}_{\ol N}\cong \scP^{\Ca I}.
$$
The element $\phi_1=\deltaunit$ is in $\scP^\Ca I,$ and it generates $\scP.$ 
So any $\Phi\in\Hom_G[\scP,\scP]$  is determined by its value on $\phi_1.$  Furthermore, $\Phi(\phi_1)\in\scP^I$.

\bigskip
Conversely,  $\phi\in\Hom_{A_0}[\sig_0,\scP_{\ol N}]\cong \scP_{\ol
  N}^{A_0}\cong\scP^{\Ca I}$ gives rise to $\Phi\in\Hom_G[\scP,\scP]$  by the
relation
\[
\Phi(\delta_{\Ca I^-B,\delta_{A_0,\one}})=\phi.
\]
The map
\[
h\in\Ca H\mapsto \Pi(h)\big(\deltaunit\big)
\]
is an isomorphism between $\Ca H$ and $\scP^{\Ca I}.$ Let $h_\psi\in
\Ca I$ be the element in $\Ca H$ corresponding to $\psi.$ Then if $\Phi(\deltaunit)=\phi,$ 
\[
\Phi[\psi]=\Phi[\Pi(h_\psi)(\deltaunit)]=\Pi(h_\phi)\Phi[\deltaunit]=\Pi(h_\psi)\phi.
\]
Now let $\phi_1,\phi_2\in \scP_{\ol N}^{A_0}.$ Then
\begin{align*}
&(\Phi_1\circ\Phi_2)(\deltaunit)=
\Phi_1[\Pi(h_{\phi_2})(\deltaunit)]=
\Pi(h_{\phi_2})\Phi_1[\deltaunit]=\\
&=\Pi(h_{\phi_2})\Pi(h_{\phi_1})(\deltaunit)
=\Pi(h_{\phi_1})\cdot\Pi(h_{\phi_2})(\deltaunit).
\end{align*}

\end{proof}

\begin{remark} The
opposite algebra to the Iwahori-Hecke algebra is isomorphic to itself, e.g.,
\[
T\circ_{opp}\theta=\theta^{-1}\circ_{opp} T
+(q-1)\frac{\theta-\theta^{-1}}{1-\theta_{-\al}}
\]
is equivalent to
\[
\theta\cdot T=T\cdot\theta^{-1}+(q-1)\frac{\theta
  -\theta^{-1}}{1-\theta_{-\al}}.
\]
\end{remark}

\subsection{}
 The operators $J_\al$ are defined analogously to $J$ for each simple root, integration is along the root subgroup $N_\al.$ The operators satisfy the formula analogous to \ref{p:adjointness}. By specializing to $\nu\in\widehat A$ unramified, we can prove the following result.
 Define 
\begin{equation}
F(\Theta)=(q-1)\frac{1}{1-\Theta^{-1}},
\end{equation}
 and write $F_\al$ for $F(\Theta_\al).$
\begin{theorem}\label{t:End-P}
  \begin{equation}
    \label{eq:6.3.1}
    T_\al:=J_\al-F_\al\in \Hom_G[\scP,\scP].
  \end{equation}
$T_\al$ and $\Theta_\al$ form a set of generators of
$\Hom[\scP,\scP]$ and satisfy the defining relations in the
Bernstein-Lusztig presentation (\cite{L1}) for the
Iwahori-Hecke algebra. 
\end{theorem} 
\begin{proof}[Sketch of proof]
Because the group is split, this reduces to a calculation in $SL(2)$. The operator $J$ has a term  which is a rational function in $\Theta_\al$ with $1-\Theta_{-\al}$ in the denominator,  and subtracting $F_\al$ removes the  singularity. 

\end{proof}

\begin{remark}
For a classical $p$-adic group $G$ and any Bernstein projective module $\C
P$, it is shown in \cite{He} that a generalization of Theorem
\ref{t:End-P} holds, namely, $\End_G[\C P]$ is naturally isomorphic to
an extended affine Hecke algebra with unequal parameters.
\end{remark}

  \begin{proposition}
    There is $f_\al\in \scH(K_\ell \backslash G/K_\ell)$ and $\tau_\al:\sig\longrightarrow\sig$ such that
    \begin{equation}
      \label{eq:6.3.2}
  <\Phi(\delta_{K_\ell w_0 P,x}),\Psi(T_\al(\delta_{K_\ell
    P,y}))>=<\Phi(\delta_{K_\ell w_0 P,x}), \Psi\big(\Pi(f_\al)(\delta_{K_\ell P,\tau_\al(y)})\big)>.
    \end{equation}
  \end{proposition}
  
  \begin{proof}
{This follows from the formula of $J_\al$ as an integral. We want $T_\al(\delta_{K_\ell P,y})=\Pi(f_\al)(\delta_{K_\ell
    P,y})$.}
 
\medskip   
For $SL(2)$, let  $K_\ell$ be the usual congruence subgroup. Let
$a:=\begin{bmatrix}\varpi&0\\0&\varpi^{-1}\end{bmatrix}.$ 
Then $\Ca IB=\Ca I^-B,$ and $a^{-\ell} K_\ell Ba^{\ell}=\Ca I^-B.$ Thus
\[
\Pi(a^{-\ell})(\delta_{K_\ell B,\al})=\delta_{\Ca Ia^{-\ell}B,a^\ell\al}=\Pi(\delta_{\Ca Ia^{-\ell}\Ca I})\delta_{\Ca IB,\al}.
\]
$T_\al$ commutes with $\Pi(\delta_{\Ca Ia^{-\ell}\Ca I})$ and $\Pi(a^{-\ell})$, and is computable on $\delta_{\Ca I B,\al}$. it can be written as convolution with a $\Ca I-$biinvariant function.
The conclusion of the calculation is that $T_\al(\delta_{K_\ell B,\al})$ can be expressed as  convolution with an element $\scT_\al\in \Ca H(\Ca I\backslash G/\Ca I)$ and composition with a $\Pi(a^{\pm\ell})$.    We can then argue as in Proposition \ref{p:starfortheta} to conclude that 
\begin{equation}
\langle\Phi,\Psi\circ T_\al\rangle=\langle\Phi\circ T_\al,\Psi\rangle.
\end{equation}
\end{proof}
We summarize the results.
\begin{theorem}\label{thm:bullet}
In the case of Iwahori fixed vectors, unramified principal series, $\Ca H:=\Hom[\scP,\scP]$ inherits a natural  star operation $\bullet$  from the unitary structure of $\scP$ satisfying
\[
\langle
\Phi,\Psi\circ\scR\rangle=\langle\Phi\circ\scR^\bullet,\Psi\rangle,
\quad \Phi,\Psi, \C R\in \Ca H.
\]
In particular,
\[
\begin{aligned}
&T_\al^\bullet=T_\al,
&\Theta^\bullet=\Theta .
\end{aligned}
\]

\end{theorem}

\section{Star operations: the graded affine Hecke algebra}\label{sec:2}

\subsection{Graded affine Hecke algebra}

We fix an $\bR$-root system $\Phi=(V,R,V^\vee, R^\vee)$. This means
that $V, V^\vee$ are finite dimensional $\bR$-vector spaces, with a
perfect bilinear pairing $(~,~): V\times V^\vee\to \bR$, where $R\subset
V\setminus\{0\},$ $R^\vee\subset V^\vee\setminus\{0\}$ are finite
subsets in bijection 
\begin{equation}
R\longleftrightarrow R^\vee,\ \al\longleftrightarrow\al^\vee,\
\text{ satisfying }(\al,\al^\vee)=2. 
\end{equation}
Moreover, the reflections
\begin{equation}
s_\al: V\to V,\ s_\al(v)=v-(v,\al^\vee)\al, \quad s_\al:V^\vee\to
V^\vee,\ s_\al(v')=v'-(\al,v')\al^\vee, \quad \al\in R, 
\end{equation}
leave $R$ and $R^\vee$ invariant, respectively. Let $W$ be the
subgroup of $GL(V)$ (respectively $GL(V^\vee)$) generated by
$\{s_\al:~\al\in R\}$. 
We  assume that the root system $\Phi$ is reduced, meaning that
$\al\in R$ implies $2\al\notin R$. 
We fix a choice of simple roots $\Pi\subset R$, and consequently,
positive roots $R^+$ and positive coroots $R^{\vee,+}.$ Often, we will
write $\al>0$ or $\al<0$ in place of $\al\in R^+$ or $\al\in (-R^+)$,
respectively. The complexifications of $V$ and $V^\vee$ are denoted by
$V_\bC$ and $V^\vee_\bC$, respectively, and we denote by $\bar{\ }$
the complex conjugations of $V_\bC$ and $V^\vee_\bC$ induced by $V$
and $V^\vee$, respectively.  {Extend $(\ ,\ )$ linearly to $V_\bC\times
V_\bC^\vee.$ Then 
\begin{equation}
\overline{(v,u)}=(\overline v, \overline{u}),\text{ for all }v\in
V_\bC,\ u\in V_\bC^\vee. 
\end{equation}
}
Let $k: \Pi\to \bR$ be a function such that $k_\al=k_{\al'}$ whenever
$\al,\al'\in \Pi$ are $W$-conjugate. Let $\bC[W]$ denote the group
algebra of $W$ and $S(V_\bC)$ the symmetric algebra over $V_\bC.$ The
group $W$ acts on $S(V_\bC)$ by extending the action on $V.$ For every
$\al\in \Pi,$  denote the difference operator by
\begin{equation}\label{e:diffop}
\Delta: S(V_\bC)\to S(V_\bC),\quad
\Delta_\al(a)=\frac{a-s_\al(a)}{\al},\text{ for all }a\in S(V_\bC).
\end{equation}

\begin{definition}\label{d:graded}
The graded affine Hecke algebra $\bH=\bH(\Phi,k)$  is the unique
associative unital algebra generated by $\bA=S(V_\bC)$ and
$\{t_w: w\in W\}$ such that  
\begin{enumerate}
\item[(i)] the assignment $t_wa\mapsto w\otimes a$ gives an
  isomorphism $\bH\cong \bC[W]\otimes S(V_\bC)$ of
  $(\bC[W],S(V_\bC))$-bimodules; 
\item[(ii)] $a t_{s_\al}=t_{s_\al}s_\al(a)+k_\al \Delta_\al(a),$
  for all $\al\in \Pi$, $a\in S(V_\bC).$
\end{enumerate}
\end{definition}

{The center of $\bH$ is $S(V_\bC)^W$ (\cite{L1}). By Schur's Lemma, the center of $\bH$
acts by scalars on each irreducible $\bH$-module. The central
characters are parameterized by $W$-orbits in $V_\bC^\vee.$ If $X$ is
an irreducible $\bH$-module, denote by $\cc(X)\in W\backslash
V_\bC^\vee$ 
its central character. By abuse of notation, we may also denote by
$\cc(X)$ a representative in $V_\bC^\vee$ of the central character of
$X$. 

If $(\pi,X)$ is a finite dimensional $\bH$-module and $\lambda\in V_\bC^\vee$, denote 
\begin{equation}
X_\lambda=\{x\in X: \text{ for every }a\in S(V_\bC),\
(\pi(a)-(a,\lambda))^nx=0,\text{ for some }n\in \mathbb N\}.
\end{equation}
If $X_\lambda\neq 0$, call $\lambda$ an $\bA$-weight of $X$. 
Let $\Omega(X)\subset V_\bC^\vee$ denote the set of
$\bA$-weights of $X$. If $X$ has a central character, it is easy to see that $\Omega(X)\subset
W\cdot \cc(X).$

\begin{definition}[Casselman's criterion]\label{d:tempered} 
Set 
$$V^+=\{\om\in V: (\om,\al^\vee)>0,\text{ for all }\al\in\Pi\}.$$
An
  irreducible $\bH$-module $X$ is called tempered if 
$$(\om,\Re\lambda)\le
  0,\text{ for all }\lambda\in\Omega(X)\text{ and all }\om\in V^+.$$ 
A tempered module is
  called a discrete series module if all the inequalities are strict.
\end{definition}
}
When the root system $\Phi$ is semisimple, $\bH$ has a particular
discrete series module, the Steinberg module $\St$. This is a
one-dimensional module, on which $W$ acts via the $\sgn$
representation, and the only $\bA$-weight is $-\sum_{\al\in\Pi} k_\al
\om^\vee_\al,$ where $\om^\vee_\al$ is the fundamental coweight
corresponding to $\al.$

\subsection{An automorphism of $\bH$}
Let $w_0$ denote the long Weyl group element. Define an assignment
\begin{equation}\label{autom}
\delta(t_w)=t_{w_0 w w_0},\ w\in W,\quad \delta(\omega)=-w_0(\omega),\
\omega\in V_\bC.
\end{equation}

\begin{lemma}\label{l:autom}
Suppose $k_{\delta(\al)}=k_\al,$ for all $\al\in\Pi.$ The assignment $\delta$ from (\ref{autom}) extends to an involutive automorphism of
$\bH.$ When $w_0$ is central in $W$, $\delta=\Id$.
\end{lemma}

\begin{proof}
It is clear that $\delta$ is an automorphism of $\bC[W]$ and it also
extends to an automorphism on $S(V_\bC)$, so it remains to check the
commutation relation in Definition \ref{d:graded}:
\begin{equation}\label{e:basiccomm}
\omega t_{s_\al}-t_{s_\al} s_\al(\omega)=k_\al (\omega,\al^\vee),\
\al\in\Pi,\ \omega\in V_\bC.
\end{equation}
Then 
$$
\begin{aligned}
\delta(\omega) \delta(t_{s_\al})=&\delta(\omega) t_{s_{\delta(\al)}}=t_{s_{\delta(\al)}} s_{\delta(\al)}(\delta(\omega))+k_{\delta(\al)} (\delta(\omega), \delta(\al)^\vee)=\\
=&t_{s_{\delta(\al)}} s_{\delta(\al)}(\delta(\omega))+k_\al(\omega,\al^\vee).
\end{aligned}
$$
Notice that we have used the fact
that $\delta(\al)\in \Pi$ if $\al\in \Pi.$ It is easy to see that $\delta(s_\al(\omega))=s_{\delta(\al)}(\delta(\omega))$.

Since $w_0^2=1,$ $\delta^2=\Id.$
\end{proof}

Thus, one may define an extended graded Hecke algebra $\bH'=\bH\rtimes\langle\delta\rangle.$

\subsection{Star operations} 

\begin{definition}
Let $\kappa:\bH\to \bH$ be a conjugate linear involutive algebra
anti-automorphism. An $\bH$-module $(\pi,X)$ is said to be
$\kappa$-hermitian if $X$ has a hermitian form $(~,~)$ which is
$\kappa$-invariant, i.e., 
$$
(\pi(h)x,y)=(x,\pi(\kappa(h))y),\quad x,y\in X,\  h\in\bH.
$$
A hermitian module $X$ is $\kappa$-unitary if the $\kappa$-hermitian
form is positive definite. 
\end{definition}

{\begin{definition}\label{d:basicstars}
Define 
\begin{align}\label{e:star}
t_{w}^\star=t_{w^{-1}},\ w\in W,\quad \omega^\star=-t_{w_0}\cdot
\overline{\delta(\omega)}\cdot t_{w_0}= (\Ad t_{w_0}\circ\delta)
(\overline{\omega}),\ \omega\in V_\bC, 
\end{align}

and
\begin{align}\label{e:bullet}
t_{w}^\bullet=t_{w^{-1}},\ w\in W,\quad \omega^\bullet=\overline\omega,\ \omega\in V_\bC.
\end{align}
\end{definition}
}
\begin{lemma}\label{l:1.3.3}
The operations $\star$ and $\bullet$ defined in (\ref{e:star}) and (\ref{e:bullet}), respectively, extend to conjugate linear algebra anti-involutions of $\bH$.
\end{lemma}

\begin{proof}
Straightforward by Lemma \ref{l:autom}.
\end{proof}
{\begin{remark}
The two star operations just defined are related as follows

\begin{equation}\label{e:relstar}
\star=(\Ad t_{w_0} \circ \delta)(h)\circ\bullet.
\end{equation}
In particular, when $w_0$ is central in $W$, they are inner conjugate
to each other.
\end{remark}
}

\begin{lemma}\label{l:conjtw}
For every $w\in W$, $\omega\in V_\bC$, 
\begin{equation}\label{e:commute}
t_w\cdot \omega\cdot t_{w^{-1}}=w(\omega)+\sum_{\beta>0,
  w(\beta)<0}k_\beta (\omega,\beta^\vee)t_{s_{w(\beta)}}. 
\end{equation}
In particular,
\begin{equation}
\omega^\star=-\overline\omega+\sum_{\beta>0}k_\beta (\overline\omega,\beta^\vee) t_{s_\beta}.
\end{equation}
\end{lemma}

\begin{proof}
This is \cite[Theorem 5.6]{BM2}.
\end{proof}

\subsection{Classification of involutions}\label{sec:inv} We define a filtration of
$\bH$ given by the degree in $S(V_\bC).$ Set $\deg t_w
a=\deg_{S(V_\bC)} a$ for every $w\in W$, and homogeneous element $a\in
S(V_\bC)$ and $F_i\bH=\text{span}\{h\in \bH: \deg h\le i\}.$ In
particular, $F_0\bH=\bC[W]$. Set $F_{-1}\bH=0.$ It is immediate from
Definition \ref{d:graded} that the associated graded algebra
$\overline\bH=\oplus_{i\ge 0} \overline \bH^i,$ where
$\overline\bH^i=F_i\bH/F_{i-1}\bH$, is naturally isomorphic to the graded Hecke
algebra for the parameter function $k_\al\equiv 0.$ 

\begin{definition}\label{d:admissible}
An automorphism (respectively, anti-automorphim) $\kappa$ of $\bH$ is called
\textit{filtered} if 
  $\kappa(F_i\bH)\subset F_i\bH,$ for all $i\ge 0.$ Notice that by
Definition \ref{d:graded}, this is equivalent with the requirement
that $\kappa(F_i\bH)\subset F_i\bH$ for $i=0,1.$ If, in addition,
$\kappa(t_w)=t_{w}$ (resp., $\kappa(t_w)=t_{w^{-1}}$), we say that
$\kappa$ is \textit{admissible}.  
\end{definition}

If $\kappa$ is a filtered automorphism, then $\kappa$ induces an
automorphism of the associated graded algebra $\overline \bH$ which
preserves that grading, i.e., $\kappa(\overline\bH^i)\subset
\overline\bH^i.$

\begin{lemma}\label{l:star-assoc-graded}
Assume the root system $\Phi$ is simple. Let $\kappa$ be an {admissible}
involutive automorphism (or anti-automorphism) of
$\overline\bH$ which preserves the grading
$\kappa(\overline\bH^i)\subset \overline\bH^i.$ Then
$\kappa(\om)=c_0\om,$ for all $\om\in V_\bC$, where $c_0$ is a constant
equal to $1$ or $-1$.
\end{lemma}

\begin{proof}We prove the statement in the case when $\kappa$ is an
  automorphism. 
{Since $\bH$ is isomorphic to the opposite algebra
  $\bH^{\mathsf{opp}}$ via the map $\tau:$ $t_{w}\mapsto t_{w^{-1}}$,
  $\omega\mapsto \omega$, the classification of anti-automorphisms
  follows by composition with $\tau.$
}

By the assumptions on $\kappa$,
\begin{equation}
\kappa(\om)=\sum_{y\in W} f_y(\om) t_y,\ \om\in V_\bC,
\end{equation}
where $f_y:V_\bC\to V_\bC$ is a linear function, for every $y\in W.$
Let $\al$ be a simple root. The commutation relation in $\overline\bH$
is $t_{s_\al}\om=s_\al(\om) t_{s_\al}$. Applying $\kappa$ to this
relation, it follows, by a simple calculation, that
$$s_\al(f_{s_\al x}(\om))=f_{xs_\al}(s_\al(\om)),\text{ for all }x\in
W.$$
In particular, setting $x=s_\al,$ we see that
\begin{equation}
s_\al(f_1(\om))=f_1(s_\al(\om)).
\end{equation}
Since the root system was assumed simple, this means that $f_1$ is a
scalar function $f_1(\omega)=c_0\omega,$ for some $c_0\in\bC.$ 

Now, we use that $\kappa$ is an involution, $\kappa^2(\om)=\om,$ which
implies $\sum_{x,y\in W} (f_x\circ f_y)(\om) t_{xy}=\om.$ Thus 
\begin{equation}
\sum_{x\in W} f_x\circ f_{x^{-1}}=\text{Id},\text{ and } f_x\circ
f_y=0,\text{ if }x\neq y^{-1}.
\end{equation}
Specializing $y=1$ in the second relation, we see that $f_x=0$ if
$x\neq 1.$ Then the first relation implies $c_0^2=1$, and this is the
claim of the lemma.
\end{proof}

\begin{proposition}\label{p:classification} Assume the root system
  $\Phi$ is simple.  If $\kappa$ is an admissible involutive automorphism or
anti-automorphism (in the sense of Definition \ref{d:admissible}),
then $$\kappa(\omega)=\omega,\text{ for all }\omega\in V,$$
or $$\kappa(\omega)=t_{w_0}\cdot \delta(\omega)\cdot t_{w_0},
\text{ for all }\omega\in V.$$ 
In particular, the only admissible conjugate
linear involutive anti-automorphisms of $\bH$ are $\star$ and
$\bullet$ from Lemma \ref{l:1.3.3}. 
\end{proposition}

\begin{proof}
As before, it is sufficient to only treat the case when $\kappa$ is an
automorphism.
The hypotheses imply that $\kappa^2=\Id$, and in addition, by
Lemma \ref{l:star-assoc-graded}, $\kappa$ must be of the form:
\begin{align*}
&\kappa(t_w)=t_w,\quad w\in W;
&\kappa(\omega)=c_0\omega+\sum_{y\in W}g_y(\omega) t_y, \quad \omega\in V_\bC,
\end{align*}
where $g_y:V_\bC\to\bC$, $y\in W,$ are
linear, {and $c_0=\pm 1$}. 
Since $\kappa$ has to preserve the commutation relation  
\[  t_{s_\al}\omega-s_\al(\omega)t_{s_\al} = k_\al (\omega,\al^\vee),\ \al\in\Pi, \omega\in V_\bC,\]
we find that
\[ 
c_0 t_{s_\al} \omega-c_0 s_\al(\omega) t_{s_\al}+\sum_{y\in W}
g_y(\omega) t_{s_\al y}-\sum_{x\in W} g_x(s_\al(\omega))t_{x
  s_\al}=k_\al(\omega,\al^\vee),\] 
or equivalently,
\begin{equation}\label{e:comm}
\sum_{y\in W}
g_y(\omega) t_{s_\al y}-\sum_{x\in W} g_x(s_\al(\omega))t_{x
  s_\al}=k_\al(1-c_0)(\omega,\al^\vee). 
\end{equation}
This implies that
\begin{equation}\label{e:gconj}
g_{s_\al y s_\al}(\omega)=g_{y}(s_\al(\omega)), \text{ for all }\al\in
\Pi, y\in W, y\neq s_\al, \text{ and }\omega\in V_\bC, 
\end{equation}
and
$$
g_{s_\al}(\omega)-g_{s_\al}(s_\al(\omega))=k_\al(1-c_0)(\omega,\al^\vee),
$$
from which one easily concludes that 
\begin{equation}\label{e:g_al}
g_{s_\al}(\al)=k_\al (1-c_0),\ \al\in\Pi.
\end{equation}
There are two cases:
\begin{enumerate}
\item $c_0=1$,
\item $c_0=-1$. 
\end{enumerate}
In case (1),  $g_y=0$ for all $y$, so $\kappa(\omega)=\omega$, $\omega\in V$.

When $w_0=-\Id,$ we note that since $\kappa(t_{w_0})=t_{w_0},$
 $\kappa':=\kappa\circ\Ad t_{w_0}=\Ad t_{w_0}\circ\kappa$ is as in
 case (1). 

{When $w_0\ne -\Id,$ recall $\delta$ the automorphism defined
 in (\ref{autom}). 
If we knew that $\kappa\circ
 \delta=\delta\circ\kappa,$ the same proof would apply, since $\delta\circ\Ad
 t_{w_0}\circ\kappa$ is of the same type as $\kappa,$ but $c_0$
 changes to $-c_0$. Since this is not clear
 to us, we prove directly that in case (2),
 $\kappa(\omega)=t_{w_0}\cdot\delta(\omega)\cdot t_{w_0}$.}

\smallskip

We first show that $g_y=0$ unless $y=s_\beta$ for some
positive root $\beta.$ If $y=1,$ relation (\ref{e:gconj}) shows that
$g_y=0,$ so assume $y\ne 1.$  The automorphism $\kappa$ must also satisfy
$\kappa(\omega_1)\kappa(\om_2)=\kappa(\om_2)\kappa(\om_1)$ for all
$\omega_1,\omega_2\in V_\bC.$ This implies that 
\begin{equation}\label{e:om1om2}
g_y(\omega_2)(\omega_1-y^{-1}(\omega_1))=g_y(\omega_1)(\omega_2-y^{-1}(\omega_2)),\text{
  for all }y\in W,\ \omega_1,\omega_2\in V_\bC.
\end{equation} 
If $\la_1,\la_2$ are eigenvalues of $y^{-1},$ then for
$\om_1\in V_{\la_1}\ \om_2\in V_{\la_2},$
\begin{equation}
  \label{eq:relation}
g_y(\om_1)(1-\la_2)\om_2=g_y(\om_2)(1-\la_1)\om_1.
\end{equation}
Set $\la_1=1.$ Then
\[
g_y(\om_1)(1-\la_2)\om_2=0 \text{ for any } \om_2\in V_{\la_2}.
\]  
Because $y^{-1}\ne 1,$ it has an eigenvalue $\la_2\ne 1,$ so $g_y$ is
0 on the $1-$eigenspace of $y^{-1}.$ Similarly, relation
(\ref{eq:relation}) implies that if $\la\ne 1,$ any 
$\om_1,\om_2\in V_\la$ must be multiples of each other. So $\dim
V_\la\le 1$ for any $\la\ne 1.$ 

Because $y$ is an automorphism of the real space
$V,$ if $\la$ is an eigenvalue, so is $\ovl\la.$  From relation
(\ref{eq:relation}), we see that 
unless $\la=\ovl\la,$ $g_y=0$ on these
eigenspaces. The only remaining case, when $g_y\ne 0$, is when  $y^{-1}$ has
eigenvalues $\pm 1,$ and the $-1-$eigenspace has dimension $1$. It
follows that $g_y=0$ unless $y=s_\beta$ for a root $\beta.$ 

\medskip
We specialize $y=s_\beta$, for $\beta\in R^+.$ Then
\[
g_{s_\beta}(\omega_2)(\omega_1,\beta^\vee)\beta=
g_{s_\beta}(\omega_1)(\omega_2,\beta^\vee)\beta,\ \omega_1,\omega_2\in V_\bC,
\]
and therefore $g_{s_\beta}(\omega)=c_\beta(\omega,\beta^\vee),$ for
some $c_\beta\in\bC.$ When $\beta=\al\in\Pi,$ (\ref{e:g_al}) with
$c_0=-1$, implies that $c_\al=k_\al$. If $\beta$ is not a simple root,
we can use (\ref{e:gconj}) inductively to check that
$c_\beta=k_\beta.$ 
 
\end{proof}

\begin{remark}
There may be many more (up to inner conjugation) filtered
automorphisms $\kappa$ that preserve, but are not the identity on
$W$. Every filtered automorphism $\kappa$ {induces}  an
automorphism of $\bC[W]$, so a first question would be to classify the
group of outer automorphisms of $\bC[W]$, a subgroup of which is
$\mathsf{Out}(W)$, and this can be nontrivial (e.g., when $W=S_6$,
$\mathsf{Out}(S_6)=\bZ/2\bZ$). But if we require that $\kappa$
preserves the root reflections, then $\kappa$ is obtained from one of the two
automorphisms in Proposition \ref{p:classification} by composition
with an automorphism of $\bH$ coming from the root system.    
\end{remark}

{}

\section{Relation between signatures}
{In this section, we discuss the relation between the
  signature characters for $\star$ and $\bullet$ of simple hermitian $\bH$-modules}.
\subsection{}
Let $\bH'=\bH\rtimes\langle\delta\rangle$ be the extended graded Hecke
algebra and $(\pi,X)$ a module for $\bH'$. Then, $X$ has a
$\bullet$-invariant form if and only if it has a $\star$-invariant
form, see \cite[Lemma 3.1.1]{BC4} and the relation between the forms
is 
\begin{equation}
\langle x,y\rangle_\star=\langle x,\pi(t_{w_0}\delta) y\rangle_\bullet.
\end{equation}
For example, this applies to the case when $X$ is a simple
$\bH$-module with real central character. In that case, let
$(\mu,U_\mu)$ be a lowest $W$-type of $X$, and extend $X$ to a
$\bH'$-module, as we may, by normalizing the action of $\delta$ so
that $\pi(t_{w_0}\delta)$ acts on $\mu$ by the identity. (Since
$t_{w_0}\delta$ is central in $W'=W\rtimes\langle\delta\rangle$, a
priori, it acts on $\mu$ by $\pm\Id$ depending how $\mu$ is extended
to a $W'$-type.) 

Define the elements
\begin{equation}
\wti\om=\frac 12(\om-\om^\star)=\om-\frac
12\sum_{\beta>0}\Delta_\beta(\om) t_{s_\beta},\quad\om\in V. 
\end{equation}
These elements satisfy:

\begin{enumerate}
\item $\wti\om^\star=-\wti\om$, $\wti\om^\bullet=\wti\om$;
\item $t_w\wti\om t_{w^{-1}}=\wti{w(\om)}$;
\item $[\wti\om_1,\wti\om_2]=-[T_{\om_1},T_{\om_2}]\in \bC[W]$, where $T_\om=\frac 12\sum_{\beta>0}\Delta_\beta(\om) t_{s_\beta}.$
\end{enumerate}

Define the following increasing filtration on $X$:
\begin{equation}
\C F^kX=\text{span}\{\pi(\wti\om_1\cdot\dots\cdot\wti\om_j)u: u\in
U_\mu,\ \om_i\in V,\ j\le k\},\quad k\ge 0.  
\end{equation}
This filtration is $W'$-invariant by (2) above, and obviously finite
if $X$ is finite dimensional. It depends on the chosen lowest $W$-type
$\mu$. One can define this filtration by starting with any
$W$-invariant subspace of $X$ in degree $0$, for example, by replacing
$V_\mu$ with the sum of lowest $W$-types. 

Let $\oplus\overline{\C F}^kX$ be the associated graded object, each
$\overline {\C F}^kX$ is a $W$-module. 
Since
\[\pi(t_{w_0}\delta)\pi(\wti\om_1\cdot\dots\cdot\wti\om_j)u=(-1)^k\pi(\wti\om_1\cdot\dots\cdot\wti\om_j)\pi(t_{w_0}\delta)u=(-1)^k\pi(\wti\om_1\cdot\dots\cdot\wti\om_j)u, 
\]
$t_{w_0}\delta$ acts by $(-1)^k$ on $\overline {\C F}^kX$. This means
that $\langle \overline{\C F}^kX,\overline{\C F}^\ell
X\rangle_\bullet=0$ if $k\not\equiv\ell$ (mod $2$) and moreover, we
have the following relation. 
\begin{lemma} If $x\in X$, let $k(x)$ be the integer such that 
{$0\ne\ovl{x}\in \overline{\C F}^{k(x)}X$}. Then
$$
\langle x, y\rangle_\star=\begin{cases}\langle x,y\rangle_\bullet=0,
  &\text{if } k(x)\not\equiv k(y)\text{ (mod $2$),}\\ 
(-1)^{k(x)}\langle x,y\rangle_\bullet, &\text{if } k(x)\equiv
k(y)\text{ (mod $2$).} \end{cases}
$$ 
\end{lemma}

\subsection{}\label{sec:3.2}Denote 
\begin{equation}
\overline X_0=\sum_{k\text{ even}} \overline{\C
  F}^kX \text{ and }\overline X_1=\sum_{k\text{ odd}} \overline{\C
  F}^kX.
\end{equation}
 Notice that $\overline X_0$ and $\overline X_1$ are the $+1$
and $-1$ eigenspaces of $t_{w_0}\delta$ in $X$, so they do not depend
on the chosen filtration. The previous lemma implies that a necessary
condition for the module $X$ to be $\star$-unitary is that the
$\bullet$-form be positive definite on $\overline X_0$ and 
negative definite on $\overline X_1.$

 Let $\bH_\ev$ be the subalgebra of $\bH$ generated by $W$ and
 $\{\wti\om_1\wti\om_2:\om_1,\om_2\in V\}$ and
 $\bH'_\ev=\bH_\ev\rtimes\langle\delta\rangle$. Then $\overline X_0$
 and $\overline X_1$ are both $\bH'_\ev$-modules, and with the
 inherited $\bullet$-form, they are $\bullet$-unitarizable
 $\bH'_\ev$-modules. Notice also that if $X$ is a simple
 $\bH'$-module, then $\overline X_0$ and $\overline X_1$ are simple
 $\bH'_\ev$-modules.  

 In conclusion:

\begin{lemma}
A necessary condition for the simple $\bH'$-module $X$ to be
$\star$-unitary is that $\overline X_0$ and $\overline X_1$ be
$\bullet$-unitarizable simple $\bH'_\ev$-modules (or zero). 
\end{lemma}

\begin{example}
Let $\bH$ be the graded algebra of type $A_1$ with generators $t$ and $\om$: $t\om+\om t=2$. Then $\wti\om=\om-t$ and $\bH_\ev$ is generated by $t$ and $\om^2$. Since $t\om^2=\om^2t$, the simple $\bH_\ev$-modules are one-dimensional of the form $X(\triv,\lambda)$ or $X(\sgn,\lambda)$, where the restriction to $W$ is $\triv$ or $\sgn$, respectively, and $\om^2$ acts by $\lambda.$ Suppose $\lambda$ is real and let $x_\lambda$ be a generator of such a module $X$. Then we can define a positive definite $\bullet$-invariant form on $X$ by setting $\langle x_\lambda,x_\lambda\rangle_\bullet=1$.  
\end{example}

\begin{example}
Suppose $(\pi,X)$ is a simple tempered $\bH$-module with real central
character. Let $\fg$ be the complex Lie algebra attached to the root
system and $G=\Ad\fg$.  By \cite{KL,L2}, there exists a nilpotent
element $e\in \fg$ and $\psi\in\widehat {A_G(e)}$ of Springer type
such that $$X|_W= H^*(\C B_e)^\psi=\sum_{i=0}^{d_e} H^{2i}(\C
B_e)^\psi.$$ To emphasize the connection write $X(e,\psi)$ for
$X$. Then $X(e,\psi)$ has a unique lowest $W$-type, namely
$\mu(e,\psi)= H^{2d_e}(\C B_e)^\psi$, and we define the filtration
accordingly. One can define an action of $\delta$ on $H^*(\C
B_e)^\psi$ (see \cite[section 4]{CH} or \cite{BeMi}), which makes
$H^*(\C B_e)^\psi$ into a $W'$-module and 
\begin{equation}
\tr(ww_0\delta,H^{2i}(\CB_e)^\psi)=(-1)^i\sgn(w_0)\tr(w,H^{2i}(\CB_e)^\psi).
\end{equation}
Moreover, this action is compatible with the $\bH$-action on
$X(e,\psi)$ \cite[section 6.4]{CH}, making $X(e,\psi)$ into an
$\bH'$-module. Thus, once we normalized the action so that $\delta$
acts by $\Id$ on $\mu(e,\psi)$, we have 
\begin{equation}
\overline X_0=\sum_{0\le i\le d_e,~i\equiv d_e \text{ (mod 2)}}
H^{2i}(\C B_e)^\psi\text{ and }\overline X_1=\sum_{0\le i\le
  d_e,~i\not\equiv d_e \text{ (mod 2)}} H^{2i}(\C B_e)^\psi. 
\end{equation}
\end{example}

\begin{example}\label{one-W-type}
If $\overline X_0=X$ then $X$ must be a one-$W$-type in the sense of
\cite{BM}. The one-$W$-type modules are the only simple $\bH$-modules
with real central character which are unitary with respect to both
$\bullet$ and $\star$ operations. This follows from an argument which
is essentially in \cite[Proposition 2.3]{BM}, see  \cite[Proposition
3.1.1]{CM}. 
\end{example}

\section{Invariant forms on spherical principal series}
\label{sec:dirac} 

\subsection{Spherical principal series} In this section, we define
$\star$- and $\bullet$-invariant
hermitian forms on spherical principal series $\bH$-modules (when such forms exist). 

\smallskip

Every element $h\in\bH$ can be written uniquely as $h=\sum_{w\in W}t_w
a_w,$ $a_w\in S(V_\bC).$ Define the $\bC$-linear map
$$\ep_A:\bH\to S(V_\bC),\quad \ep_A(h)=a_1.$$ 
If $\nu\in V_\bC^\vee$, let $\bC_\nu$ denote the character of
$S(V_\bC)$ given by evaluation at $\nu$. For $a\in \bH$, denote by
$a(\nu)$ the evaluation of $a$ at $\nu.$  The spherical principal series
with parameter $\nu$ is $$X(\nu)=\bH\otimes_{S(V_\bC)} \bC_\nu.$$

If $\kappa$ is any conjugate linear anti-involution of $\bH$, and $L,R$
are arbitrary elements of $\bH$, and $\nu'\in V_\bC^\vee$,  the assignment
\begin{equation}
\langle h_1,h_2\rangle_{L,R}=\epsilon_A(L \kappa(h_2)h_1 R)(\nu'),\quad
h_1,h_2\in\bH,  
\end{equation}
defines a $\kappa$-invariant (not necessarily hermitian) pairing on
$\bH$ viewed as an $\bH$-module under left multiplication. We will
omit the subscript $L,R$ from the notation. For such a
form to descend to a $\kappa$-invariant hermitian form on $X(\nu)$, it
must satisfy: 
\begin{enumerate}
\item[(H1)] $\langle h_1 a,h_2\rangle= a(\nu)\langle h_1,h_2\rangle$,
  for all $a\in S(V_\bC)$; 
\item[(H2)] $\langle h_1 ,h_2 a\rangle= \overline{a(\nu)}\langle
  h_1,h_2\rangle$, for all $a\in S(V_\bC)$; 
\item[(H3)] $\langle h_1,h_2\rangle=\overline{\langle h_2,h_1\rangle}.$
\end{enumerate} 
Of course, (H1) and (H3) imply (H2), but in practice it will be
convenient for us to check (1) and (2) first, which will then reduce
the verification of (3) on the basis $\{t_w\in W\}$ of $X(\nu)$.

\smallskip
For every $s_\al\in W,$ $\al\in\Pi$, define 
\begin{equation}
R_{s_\al}=(t_{s_\al}\al-k_\al)(\al-k_\al)^{-1}.
\end{equation}
As it is well known, the elements $R_{s_\al}$ satisfy the braid
relations, therefore one can define $R_x$, $x\in W$, as a product,
using a reduced expression of $x$. The main property of $R_x$ is that
\begin{equation}\label{e:aR}
a\cdot R_x= R_x\cdot x^{-1}(a),\text{ for all }x\in W,\ a\in S(V_\bC).
\end{equation}

{We show (H1)-(H3) for $\kappa=\bullet$ and the pairing
\begin{equation} \label{e:bullet-minimal} 
\langle h_1,h_2\rangle_\bullet:=\ep_A(t_{w_0}h_2^\bullet h_1R_{w_0})(w_0\nu).
\end{equation}  }
Let { $$\C R_\al:=(t_\al\al-{k_\al})({k_\al+\al})^{-1},$$
and for $x=s_{\al_1}\dots s_{\al_k},$ define $\C R_x=\prod\C
R_{\al_i}.$ The $\C R_x$ have the same commutation properties as the
$R_x,$ and
\begin{equation}
  \label{eq:cbullet}
\C R_x^\bullet=(-1)^{\ell(x)}\C
R_{x^{-1}}\prod_{x^{-1}\al<0}\frac{k_\al+\al}{k_\al -\al}.  
\end{equation}
Let
\[
V_{\reg}^\vee:=\{\nu\in V_\bC^\vee\ :\ (\al,\nu)\ne 0 \text{ for any
} \al\in R^+ \}.  
\] 
For $\nu\in V_{\reg}^\vee,$} a basis of $X(\nu)$ is given by 
\begin{equation}
  \label{eq:basis}
\{\C R_x\otimes \one_\nu\}_{x\in W}.  
\end{equation}
Notice that $\C R_x$ is not in $\bH$, but in $\hat\bH.$ However it
makes sense to express 
$\C R_x=\sum t_y a^x_y$ with $a_y^x\in \CO(V_\bC),$ and
then evaluate at $\nu.$ The fact that $\nu\in V_{\reg}^\vee$ allows one
to solve for the $t_x\otimes\one_\nu$ in terms of the $\C R_x\otimes
\one_\nu$; so indeed (\ref{eq:basis}) is a basis. ({Note 
  that we have assumed that $k_\al >0.$})

\begin{lemma}\label{l:A-weights}
The vector $\C R_x\otimes \one_\nu$ is  an
$\bA$-weight vector of $X(\nu)$ with weight $x\nu$.
\end{lemma}
\begin{proof}
Since $a\cdot \C R_x=\C R_x \cdot x^{-1}(a)$, $a\in
S(V_\bC)$, it
follows that in $X(\nu),$ $a\cdot (\C R_x\otimes \one_\nu)=a(x\nu) (\C
R_x\otimes\one_\nu).$ 
\end{proof}
 
\smallskip

We show that (H1)-(H3) hold for (\ref{eq:basis}) and
$\nu\in V_{\reg}^\vee.$ Since the relations (and the change of basis matrices to
the $t_x$) are rational in $\nu,$ and
$V_{\reg}^\vee$ contains an open set in $V_\bC^\vee,$ they will hold in general.
 
{The first identity holds by (\ref{e:aR}):
\begin{equation*}
\langle h_1a,h_2\rangle_\bullet=\langle h_1,h_2\rangle_\bullet a(\nu).
\end{equation*}
For the second identity,
\begin{equation*}
\langle \C R_x,\C R_ya\rangle_\bullet=\langle \C R_x,\C R_y\rangle_\bullet
(w_0x^{-1}y)(a^\bullet)(w_0\nu)=\langle \C R_x,\C R_y\rangle_\bullet
(x^{-1}y)(a^\bullet)(\nu).
\end{equation*} 
Suppose $x=y$. Then this formula implies (H2) (with $h_1=h_2=\C R_x$) if
and only if $a^\bullet(\nu)=\overline{a(\nu)}$ which is equivalent
to $\nu=\overline\nu,$ i.e., $\nu\in V^\vee.$ 

Suppose $x\neq y.$ We show that each of the two sides of (H2) are zero
because $\ep_A(t_{w_0}\C R_zR_{w_0})=0$ unless $z=1$: 
\[
\begin{aligned}
&\ep_A\left(t_{w_0}(\C R_y a)^\bullet\C R_xR_{w_0}\right)=\ep_A\left(t_{w_0}a(-1)^{\ell
(y)}\C R_{y^{-1}}\prod_{y^{-1}\al<0}\frac{k_\al
  +\al}{k_\al -\al} \C R_xR_{w_0}\right)=\\
&=\ep_A\left(t_{w_0}\C R_{y^{-1}x}R_{w_0}\right)\cdot(-1)^{\ell (y)}(w_0x^{-1}y)(a)\prod_{y^{-1}\al<0}\frac{k_\al
  +w_0x^{-1}\al}{k_\al-w_0x^{-1}\al}=0, \text{ and}\\
&\ep_A\left(t_{w_0}\C R_y^\bullet\C R_xR_{w_0}\right)
a=\ep_A\left(t_{w_0}\C R_{y^{-1}x}R_{w_0}\right)\cdot (-1)^{\ell(x)} \prod_{y^{-1}\al<0}\frac{k_\al
  +w_0x^{-1}\al}{k_\al -w_0x^{-1}\al}a=0.
\end{aligned}
\]
So (H2) is verified.

We also record the
formula
\begin{equation}
  \label{eq:bprod}
\begin{aligned}
\langle\C R_x\otimes\one_\nu,\C
R_x\otimes\one_\nu\rangle_\bullet&=(-1)^{\ell(x)}\left(\prod_{\al>0}\frac{\al}{k_\al  -\al}\cdot
\prod_{x^{-1}\al<0}\frac{k_\al
  -\delta(x^{-1}\al)}{k_\al +\delta(x^{-1}\al)}\right) \left(w_0\nu\right) \\
&=(-1)^{|R|}\prod_{\al>0}\frac{\langle\al,\nu\rangle}{\langle\al,\nu\rangle+k_\al}\prod_{x\al<0}\frac{\langle\al,\nu\rangle-k_\al}{\langle\al,\nu\rangle+k_\al}.
\end{aligned}
\end{equation}
The equivalence of the two formulas can be easily seen by the substitution $x^{-1}\al\mapsto\al$ in the second product. Notice that the factor $(-1)^{|R|}\prod_{\al>0}\frac{\langle\al,\nu\rangle}{\langle\al,\nu\rangle+k_\al}$ is independent of $x$, so we may divide the form uniformly by it. The resulting normalized hermitian form has the property that $$\langle \C R_1\otimes \one_\nu,\C R_1\otimes\one_\nu\rangle_\bullet=1.$$

When $\nu$ is dominant, $k_\al +\langle\al,\nu\rangle>0$, so the denominator does not vanish, and it is always
positive (we have assumed $k_\al >0$).

} 

\smallskip

The arguments also imply that $\langle h_2 ,h_1\rangle_\bullet =\ovl{\langle
  h_1,h_2\rangle_\bullet}$ for {
  $h_1,h_2\in\{R_x\otimes\one_\nu\}_{x\in W}$, so also in general. In
  conclusion, we have proved the following result.
\begin{proposition}\label{p:bullet-form-minimal}
The form 
$$\langle h_1,h_2\rangle_\bullet:=\ep_A(t_{w_0}h_2^\bullet h_1R_{w_0})(w_0\nu)$$
defines a $\bullet$-invariant
hermitian form on $X(\nu)$ if and only if $\overline\nu=\nu$, i.e.,
$\nu\in V^\vee.$
\end{proposition}
The case of $\star$ follows by formal manipulations. Set 
\begin{equation}\label{e:star-minimal}
\langle h_1, h_2\rangle_\star=\ep_A\left(h_2^\star h_1  R_{w_0}\right)(w_0\nu).
\end{equation}
The relation between the forms is
\begin{equation}\label{e:relation-stars}
\begin{aligned}
\langle h_1, h_2\rangle_\star&=\ep_A\left(h_2^\star h_1  R_{w_0}\right)(w_0\nu)=
\ep_A\left(t_{w_0}\delta(h_2)^\bullet t_{w_0}h_1
  R_{w_0}\right)(w_0\nu)\\
&=\langle
t_{w_0}h_1,\delta (h_2)\rangle_\bullet.
\end{aligned}
\end{equation}

We also note the following formulas for the signatures.
\begin{proposition} Write $R_{w_0}=\sum_{w\in W} t_w a_w.$
\begin{enumerate}
\item The signature of $\langle\ ,\ \rangle_\bullet$ is given by the
  signature of the 
  matrix $\left\{a_{x^{-1}yw_0}\right\}_{x,y\in W}$.
\item The signature of $\langle\ ,\ \rangle_\star$ is given by the
  signature of the
  matrix $\left\{a_{x^{-1}y}\right\}_{x,y\in W}$.
\end{enumerate}
\end{proposition}
\begin{proof}
Straightforward.
\end{proof}
\begin{corollary} For every $w\in W,$
  \[
\ep_A\left( t_w R_{w_0}\right)=\ep_A\left(\delta(t_{w^{-1}})R_{w_0}\right).
\]
\end{corollary}
\begin{proof}
The left hand side is 
\[
\ep_A\left(t_{w_0}t_{w_0}t_wR_{w_0}\right),
\]  
while the right hand side is
\[
\ep_A\left(t_{w_0}t_{w^{-1}}t_{w_0}R_{w_0}\right)
\]
Evaluating at $w_0\nu,$ the left hand side is $\langle
t_{w_0},t_{w}\rangle_{\bullet,\nu}$ while the right hand side is
$\langle t_w,t_{w_0}\rangle_{\bullet,\nu}.$ The fact that the two are
equal follows from the fact that $\langle\ ,\ \rangle_{\bullet}$ is
symmetric for $\nu$ real.
\end{proof}

{

{As a consequence of the relation (\ref{e:relation-stars})
  between $\bullet$ and $\star$ forms and Proposition
  \ref{p:bullet-form-minimal}, we have the following corollary.

\begin{corollary} \label{p:star-form-minimal}
The pairing
$$\langle h_1, h_2\rangle_\star=\ep_A\left(h_2^\star h_1  R_{w_0}\right)(w_0\nu)$$
 defines a $\star$-invariant hermitian form on $X(\nu)$ if and only if $w_0\nu=-\overline\nu$.
\end{corollary}
}

\section{Positive definite forms: spherical modules}

The spherical $\star$-unitary dual of graded Hecke algebras with equal
parameters is known by \cite{Ba}, \cite{BM3}, \cite{BC3}. For Hecke
algebras with unequal parameters, the irreducible $\star$-unitary
modules that are both spherical and generic were determined in
\cite{BC2}. {The Dirac inequality \cite{BCT} is far from sufficient to 
determine the answer. In this section, we show that the Dirac inequality 
\textbf{is} sufficient to compute the spherical $\bullet$-unitary dual, at least in the case of the graded Hecke algebra with equal parameters. 
As a result, the answer, Theorem \ref{t:spherical} is much simpler than the answer for the
spherical $\star$-unitary dual in {\it loc.~cit.} Theorem \ref{t:spherical} 
 complements the results in \cite[sections 4 and 5]{O2}}. 

\subsection{The Dirac operator} We assume that the Hecke algebra $\bH$ has equal parameters $k_\al=1$. 

 We fix a $W$-invariant inner product 
$\langle~,~\rangle$ on $V$.  Let $O(V)$ denote the orthogonal group of 
$V$ with respect to $\langle~,~\rangle$. Then $W\subset O(V)$. 
Denote also by $\langle~,~\rangle$ the dual inner product on  $V^\vee.$ 
If $v$ is a vector in $V$ or $V^\vee$, we denote $|v|:=\langle v,v\rangle^{1/2}.$

Denote by $ C(V)$ the Clifford algebra defined by {$(V,\langle~,~\rangle)$}.  Precisely,  $ C(V)$ is the associative algebra
with unit generated by $V$ with relations:
\begin{equation}
\om^2=-\langle\om,\om\rangle,\quad 
\om\om'+\om'\om=-2\langle\om,\om'\rangle.
\end{equation}
Let $\mathsf{Pin}(V)$ be the Pin group, a double cover of $O(V)$ with
  projection map \newline $p:\mathsf{Pin}\longrightarrow O(V)$, and
  let $\wti W=p^{-1}(W)\subset \mathsf{Pin}(V)$ be the pin double
  cover of $W$. See
  \cite{BCT} for more details. 

If $\dim V$ is even, the Clifford algebra $C(V)$ has a unique complex simple 
module $(\gamma, S)$ of dimension $2^{\dim
    V/2}$. When $\dim V$ is odd, there are two simple inequivalent {complex}
modules $(\gamma_+,S^+),$ $(\gamma_-,S^-)$ of dimension  $2^{[\dim
  V/2]}$. 
Fix $S$ to be one of these simple modules. The choice will not play a
role in the present considerations. Endow $S$ with a positive
definite invariant Hermitian  form $\langle ~,~\rangle_{ S}$.

We call a representation $\wti\sigma$ of $\wti W$ genuine
if it does not factor through $W$. For example, $S$ is 
a genuine $\wti W$-representation.

\

Let $\{\om_i:i=1,n\}$ be an
  orthonormal basis of $V$ with respect to $\langle~,~\rangle$. 
Define (\cite{BCT}) the Casimir element of $\bH$:
$$\Omega=\sum_{i=1}^n\omega_i^2\in \bH.$$
It is easy to see that the element $\Omega$ is independent of the
choice of bases and central in $\bH$. Moreover, if $(\pi,X)$ is an irreducible
  $\bH$-module,
then $\pi(\Omega)$ acts by the scalar $\langle \cc(X),\cc(X)\rangle.$
{Note that $\langle~,~\rangle$ is extended linearly to $V_\bC$ and
$V_\bC^\vee,$ and $\cc(X)$ stands for \textbf{any} representative of the set of 
weights}.

For every $\omega\in V$, recall that we have defined
\begin{equation}\label{omtilde}
\wti\om=\frac 12(\om-\om^\star)=\om-\frac 12 \sum_{\beta>0} (\beta,\omega) t_{s_\beta}
\; \in \; \bH.
\end{equation}
It is immediate that 
\begin{equation}
\wti\omega^\star = - \wti\omega\text{ and }\wti\om^\bullet=\wti\om.
\end{equation}
The Dirac element (\cite{BCT}) is defined as
\[
\C D = \sum_i \wti \omega_i \otimes \omega_i \in \bH \otimes {C(V)}.
\]
For a finite dimensional module $X$ of $\bH$, define
the Dirac operator $D:X\otimes S\to X\otimes S$ for $X$ (and $S$) as given by the action of $\C
D$.

\subsection{Dirac inequality} Let $\kappa$ be one of the star
operations $\star$ or $\bullet.$ 

\begin{lemma}The relations
\begin{equation}
\Omega^\star=\Omega\text{ and }\Omega^\bullet=\Omega
\end{equation}
hold. Therefore, if $X$ is $\kappa$-hermitian,
\begin{equation}
\begin{aligned}
&\overline{\langle\cc(X),\cc(X)\rangle}=\langle\cc(X),\cc(X)\rangle,
\text{ or, equivalently, }\\
&\langle\cc(X),\cc(X)\rangle=|\Re\cc(X)|^2-|\Im\cc(X)|^2.
\end{aligned}
\end{equation}
\end{lemma}

\begin{proof}
For $\star$: $\Omega^\star=t_{w_0}\cdot (\overline{-w_0(\Omega)})\cdot
t_{w_0}=t_{w_0}\cdot \Omega\cdot t_{w_0}=\Omega,$ where the last
equality follows from the fact that $\Omega\in Z(\bH).$

For $\bullet$: $\Omega^\bullet=\overline\Omega=\Omega.$
\end{proof}

Suppose $X$ is a $\kappa$-hermitian $\bH$-module with invariant form $(~,~)_X$.
Then $X \otimes S$ gets the
Hermitian form $(x\otimes s, x' \otimes s')_{X \otimes S} 
= (x,x')_X  \langle s,s'\rangle_S$.  The operator $D$ is self adjoint with
respect to $(~,~)_{X \otimes S}$ if $\kappa=\star$ and it is
skew-adjoint if $\kappa=\bullet$:
\begin{equation}
D^*=D\text{ and } D^\bullet=-D.
\end{equation}
Thus a $\kappa$-hermitian $\bH$-module 
is $\kappa$-unitary only if for all $x \in X \otimes S$,
\begin{equation}
\label{e:dcriterion}
\begin{aligned}
&(D^2 x, x)_{X\otimes S} \geq 0, \qquad \text{ if } \kappa=\star,
\text{ or }\\
&(D^2 x, x)_{X\otimes S} \leq 0, \qquad \text{ if } \kappa=\bullet.\\
\end{aligned}
\end{equation}
We write $\Delta_{\wti W}$ for the diagonal embedding of $\bC[\wti W]$ into
{$\bH \otimes C(V)$} defined by
extending
$\Delta_{\wti W}(\wti w) =t_{p(\wti w)} \otimes \wti w$
linearly.

\begin{theorem}[{\cite[Theorem 3.5]{BCT}}]\label{t:dirac} The square of the Dirac
  element equals
\begin{equation}
\C D^2=-\Omega\otimes 1+\Delta_{\wti W}(\Omega_{\wti W}),
\end{equation}
in $\bH \otimes C(V)$, where
\begin{equation}\label{omWtilde}
\Omega_{\wti W}=-\frac 14\sum_{\substack{\al>0,\beta>0\\s_\al(\beta)<0}} 
 |\al^\vee| |\beta^\vee| ~\wti s_\al \wti s_\beta\in \mathbb C[\wti
 W]^{\wti W}.
\end{equation}
\end{theorem}

Denote
\begin{equation}
\wti\Sigma(X)=\{\wti\sigma\in\mathsf{Irr}~\wti W: \Hom_{\wti
  W}[\wti\sigma,X\otimes S]\neq 0\}.
\end{equation}

\begin{corollary}[Dirac Inequality]\label{c:bound}Let $X$ be a $\kappa$-unitary $\bH$-module.
\begin{enumerate}
\item If $\kappa=\star$, then $|\Re\cc(X)|^2-|\Im\cc(X)|^2\le
  \min\{\wti\sigma(\Omega_{\wti W}): \wti\sigma\in \wti\Sigma(X)\}$.
\item If $\kappa=\bullet$, then $|\Re\cc(X)|^2-|\Im\cc(X)|^2\ge
  \max\{\wti\sigma(\Omega_{\wti W}): \wti\sigma\in \wti\Sigma(X)\}$.
\end{enumerate}

\end{corollary}
  \begin{proof}
This is  an immediate corollary of Theorem \ref{t:dirac} and
(\ref{e:dcriterion}).

  \end{proof}

\subsection{Spherical modules}  Let $X(\nu)=\bH\otimes_{\bA}\bC_\nu$ be
the spherical principal series, $\nu\in V_\bC^\vee$ and let $\overline
X(\nu)$ be the unique spherical subquotient, \ie
$\Hom_W[\triv,\overline X(\nu)]\neq 0$. 

\begin{lemma}\label{l:diracbulletspher}
Suppose $\overline X(\nu)$ is $\bullet$-unitary. Then $|\Re\nu|^2\ge |\Im\nu|^2+|\rho^\vee|^2.$
\end{lemma}

\begin{proof}
Since $\Hom_W[\triv,\overline X(\nu)]\ne 0$, we have $S\in
\wti\Sigma(\overline X(\nu)).$ It is known and easy to see that
$S(\Omega_{\wti W})=|\rho^\vee|^2$. The claim then follows from
Corollary \ref{c:bound}.
\end{proof}

The parameters $\nu\in V^\vee_\bC$ for which the spherical principal series $X(\nu)$ becomes reducible are known, see \cite{Ch} for the general case. This also follows from the Kazhdan-Lusztig classification, proved in the graded affine Hecke algebra case in \cite{L2}. When the parameter $\nu$ is regular, the reducibility is a consequence of intertwining operator calculations and it goes back in the setting of $p$-adic groups to Casselman \cite{Cas2}. In the equal parameters case, the result is that $X(\nu)$ is reducible if and only if:
\begin{equation}
(\al,\nu)=\pm 1, \text{ for some }\al\in R^+.
\end{equation}

Suppose that $\nu\in V^\vee$ is dominant, i.e.,  $(\al,\nu)\ge 0,$ for all 
$\al\in \Pi.$  
The reducibility hyperplanes $\al=1$, $\al\in R^+$, define an arrangement
of hyperplanes in the dominant Weyl chamber in $V^\vee$. One open region in the
complement of this arrangement of hyperplanes is
\begin{equation}
\C C_\infty=\{\nu\in V^\vee: (\al,\nu)>1,\ \al\in\Pi\}.
\end{equation}

\begin{lemma}[{see also \cite[Theorem 4.1]{O2}}]\label{l:spherical}
If $\nu\in \overline{\C C}_\infty$, then $\overline X(\nu)$ is $\bullet$-unitary.
\end{lemma}

\begin{proof}
It is sufficient to prove that $X(\nu)$ is $\bullet$-unitary when
$\nu\in \C C_\infty$. If $\nu\in V^\vee$ is such that $(\al,\nu)\neq
\pm 1$, then $X(\nu)$ is $\bA$-semisimple with a basis of weight vectors
given by $\C R_x\otimes \one_\nu,$ see Lemma \ref{l:A-weights}. By
(\ref{eq:bprod}), the
$\bullet$-form in this basis is diagonal, and if we normalize the form
so that $\langle \C R_1\otimes\one_\nu,\C
R_1\otimes\one_\nu\rangle_\bullet=1,$ then
\begin{equation}
\langle \C R_x\otimes\one_\nu,\C
R_x\otimes\one_\nu\rangle_\bullet=\prod_{x\nu<0}\frac{(\al,\nu)-1}{(\al,\nu)+1}.
\end{equation}
Clearly, if $\nu\in\C C_\infty$, then $\langle \C R_x\otimes\one_\nu,\C
R_x\otimes\one_\nu\rangle_\bullet>0$ for all $x\in W$.
\end{proof}

\begin{theorem}\label{t:spherical}
Suppose $\nu$ is dominant. 
\begin{enumerate}
\item If $\Im\nu=0$, then $\overline X(\nu)$ is $\bullet$-unitary if
  and only if $\nu\in \overline{\C C}_\infty.$
\item If $\Im\nu\neq 0,$ then $\overline X(\nu)$ is not $\bullet$-unitary.
\end{enumerate}
\end{theorem}

\begin{proof} (1) Suppose first that $\Im\nu=0.$ By Lemma
  \ref{l:spherical}, 
  $\overline X(\nu)$ is $\bullet$-unitary also for $\nu\in
  \overline{\C C}_\infty$.

For the converse, notice that by Lemma \ref{l:diracbulletspher}, $|\nu|\ge |\rho^\vee|$ is a
necessary condition for $\overline X(\nu)$ to be
$\bullet$-unitary. 
Let $B(0,|\rho^\vee|)\subset V^\vee$ be the open ball of radius
$|\rho^\vee|$ centered at the origin. Let $\C C$ be an arbitrary cell in the arrangement of
hyperplanes $\al=1$ in the dominant Weyl chamber of $V^\vee$. Since
the signature of the hermitian form of $\overline X(\nu)$ is the same
for all $\nu\in \C C$ (see \cite[Theorem 2.4]{BC2}), a necessary condition for
$\bullet$-unitarity in $\C C$ is that
\begin{equation}
\overline{\C C}\cap B(0,|\rho^\vee|)=\emptyset.
\end{equation}
We claim that this condition only holds when $\C C\subset \overline
{\C C}_\infty.$ 

The cell $\C C$ is characterized by the sets:
\begin{equation*}
\begin{aligned}
&J_0(\C C)=\{\al\in R^+: (\al,\nu)=1, \forall\nu\in \C C\},\\
 & J_+(\C C)=\{\al\in R^+:
(\al,\nu)>1, \forall\nu\in \C C\}, \  J_-(\C C)=\{\al\in R^+:
(\al,\nu)<1,\forall\nu\in\C C\}.
\end{aligned}
\end{equation*}

Suppose $\C C\not\subset\overline{\C C}_\infty.$ Then $J_-(\C C)\cap
\Pi\neq\emptyset.$ List the simple roots as
$\{\al_1,\dots,\al_m,\dots,\al_n\}$, $m<n$,
such that $\{\al_1,\dots,\al_m\}\subset J_0(\C C)\cup J_+(\C C)$ and
$\{\al_{m+1},\dots,\al_n\}\subset J_-(\C C).$ Let
$\{\om_1^\vee,\dots,\om_n^\vee\}\subset V^\vee$ be the fundamental
coweights, and write $\nu=\sum_{i=1}^n c_i\om_i^\vee.$ Then
\[ c_i\ge 1,\ i=1,m,\quad c_j<1,\ j=m+1,n.\]
Deform $c_i\to 1$, $i=1,m$ to get the point $\nu'=\sum_{i=1}^n
c_i'\om_j^\vee$, where $c_i'=1$ if $i=1,m$ and $c_i'=c_i$ if $i=m+1,n$. Notice that if
$(\beta,\nu)<1$, then $(\beta,\nu')<1$ also, and if $(\gamma,\nu)\ge
1$ then $(\gamma,\nu')\ge 1$ again. This means that $\nu'\in
\overline{\C C}.$ But now:
\begin{equation}
\langle\nu',\nu'\rangle=\sum_{i,j}c_i' c_j'\langle\om_i^\vee,\om_j^\vee\rangle<\sum_{i,j}\langle\om_i^\vee,\om_j^\vee\rangle=\langle\rho^\vee,\rho^\vee\rangle,
\end{equation}
and this means that $\nu'\in B(0,|\rho^\vee|)$, contradiction. We have
used here implicitly that $\langle\om_i^\vee,\om_j^\vee\rangle\ge 0.$ Thus
claim (1) is proven.

\

For claim (2), fix $b\in \sqrt{-1} V^\vee\setminus\{0\}$, and set $\Pi_b=\{\al\in\Pi:
(\al,b)=0\}\subsetneq\Pi.$ Let $R_b$ be the root subsystem defined by
$\Pi_b$ with positive roots $R_b^+.$ Let $\nu\in V^\vee$ dominant and we look at the
principal series $X(\nu+b).$

This is reducible if and only if there
exists $\beta\in R_b^+$ such that $(\beta,\nu)=1.$ Repeating the
argument above with $R_b^+$ in place of $R^+$, we find that every cell
$\C C$ contains in its closure a point $\nu$ such that $|\nu|\le
|\rho_b^\vee|.$ But now, the Dirac inequality in Lemma
\ref{l:diracbulletspher} gives the necessary condition
\begin{equation}
|\nu|^2\ge |b|^2+|\rho^\vee|^2>|\rho_b^\vee|^2,
\end{equation}
and this is not satisfied.
\end{proof}

\subsection{Dirac cohomology} Let $X$ be a finite dimensional module
of $\bH$. The Dirac cohomology of $X$ (with respect to the fixed spin
module $S$) is
\begin{equation}
H^D(X)=\ker D/\ker D\cap \im D.
\end{equation}
We say that $X$ has nonzero Dirac cohomology if $H_D(X)\neq 0$ for a
choice of spin module $S$.

Recall that when $X$ is $\kappa$-unitary, then the operator $D$ is
self-adjoint when $\kappa=\star$ and skew-adjoint when
$\kappa=\bullet.$ This implies that if $X$ is $\kappa$-unitary, then
\begin{equation}\label{H_D=kerD}
H^D(X)=\ker D.
\end{equation}
Theorem \ref{t:spherical} says, in particular, that every irreducible
subquotient of $X(\rho^\vee)$ is $\bullet$-unitary, so for all these
subquotients, equation (\ref{H_D=kerD}) applies.

The classification of irreducible subquotients of $X(\rho^\vee)$ is
well known. In the setting of $p$-adic groups, it is due to Casselman \cite{Cas3}. Each standard module with central character
$\rho^\vee$ is of the form
\begin{equation}
X_M=\bH\otimes_{\bH_M} (\St\otimes \bC_{\nu_M}),\text{ where }\nu_M=\rho^\vee-\rho^\vee_M,
\end{equation}
 for $M\subset\Pi$. Let $\overline X_M$ be the Langlands quotient of
 $X_M.$ Here $\bH_M$ denotes the graded Hecke algebra defined by the root system $(V,R_M,V^\vee,R_M^\vee)$, where $R_M$ is the subset of roots spanned by $M$. (We embed $\bH_M$ as a subalgebra of $\bH$ in the natural way.) Thus we have a one-to-one correspondence between irreducible
 modules with central character $\rho^\vee$ and subsets of simple
 roots $\Pi.$ In particular, there are $2^n$, $n=|\Pi|$, distinct
 simple modules with central character $\rho^\vee.$ Moreover, the
 following known character formula holds:
\begin{equation}\label{e-alt}
\overline X_M=\sum_{M\subseteq J\subseteq \Pi} (-1)^{|J|-|M|} X_J.
\end{equation}
For example, this formula follows via the Kazhdan-Lusztig
classification (and conjecture) \cite{L2} since all the $G(\rho^\vee)$-orbits on
$\fg_1(\rho^\vee)$ have smooth closure.

\begin{proposition}\label{p:wedge} Suppose $V_\bC^W=0$, i.e, the root
  system is semisimple. For every $M\subset\Pi,$ we have $\dim
  \Hom_W[\overline X_M,\bigwedge^* V_\bC]=1.$
\end{proposition}

\begin{proof}
The character formula (\ref{e-alt}) implies 
\begin{equation}\label{5.4.5}
\begin{aligned}
\Hom_W[\overline X_M,{\bigwedge}^*V]&=\sum_{J\supseteq M}
(-1)^{|J|-|M|}\Hom_W[X_J,{\bigwedge}^* V_\bC]\\
&=\sum_{J\supseteq M}
 (-1)^{|J|-|M|}\Hom_{W_J}[\sgn, {\bigwedge}^*(V_\bC)|_{W_J}],
\end{aligned}
\end{equation}
by Frobenius reciprocity. Let $V_{J,\bC}$ be the $\bC$-span of the
roots in $J$, and $V_{J,\bC}^\perp$ the orthogonal complement, so that
$V=V_{J,\bC}\oplus V_{J,\bC}^\perp.$ Notice that $W_J$ acts by the
reflection representation on $V_{J,\bC}$ and trivially on
$V_{J_\bC}^\perp.$ Since $${\bigwedge}^k (V_\bC)|_{W_J}=\bigoplus_{i=0}^k
{\bigwedge}^i V_{J,\bC}\otimes {\bigwedge}^{k-i}V_{J,\bC}^\perp,$$
we see that 
\begin{equation}
\dim\Hom_{W_J}[\sgn,{\bigwedge}^k(V_\bC)|_{W_J}]=\begin{cases}
  \dim{\bigwedge}^{k-j} V_{J,\bC}^\perp, & k\ge |J|,\\
0, &k<|J|.
\end{cases}
\end{equation}
This means that
$\dim\Hom_{W_J}[\sgn,{\bigwedge}^*(V_\bC)|_{W_J}]=\dim {\bigwedge}^*
V_{J,\bC}^\perp=2^{|\Pi|-|J|}$, and therefore
\begin{equation}\label{5.4.7}
\Hom_W[\overline X_M,{\bigwedge}^*V_\bC]=\sum_{M\subseteq J\subseteq\Pi}
(-1)^{|J|-|M|} 2^{|\Pi|-|J|}=1.
\end{equation}
\end{proof}

\begin{remark}\ 
\begin{enumerate}
\item An alternative proof of Proposition \ref{p:wedge} is as follows. Use 
  $${\bigwedge}^*V_\bC=\begin{cases} S\otimes S, &\text{ if }\dim
    V\text{ is even},\\ \frac 12 (S^++S^-)\otimes (S^++S^-),&\text{
      if }\dim V \text{ is odd}, \end{cases}$$
 and rewrite  in
equation (\ref{5.4.5}), $$\Hom_W[X_J,{\bigwedge}^*V_\bC]=a\Hom_{\wti
  W}[X_J\otimes \C S,\C S],$$ where $\C S=S$, $a=1$, if $\dim V$ is even, and
$\C S=S^++S^-$, $a=\frac 12$, if $\dim V$ is odd. Then we can apply
\cite[Lemma 2.6.2]{BC-wallach} which, in particular, gives the
dimension of this space, and we arrive at formula (\ref{5.4.7}).

\item Notice  that Proposition \ref{p:wedge} says that every simple $\bH$-module
$\overline X_M$ at
$\rho^\vee$ contains one and only one $W$-type that appears in 
$\bigwedge^* V_\bC.$ 
It is worth recalling that when the root system is simple, every 
$\bigwedge^k V_\bC$  is in fact irreducible as a $W$-representation and that 
any two distinct exterior powers are non-isomorphic, see 
\cite[Theorem 5.1.4]{GP}.

\item Since
  $\dim\Hom_W[X(\rho^\vee),\bigwedge^*V_\bC]=\dim\Hom_W[\bC[W],\bigwedge^*V_\bC]=\dim\bigwedge^*V_\bC=2^{|\Pi|}$
 and this is also the number of distinct simple $\bH$-modules with
  central character $\rho^\vee$, Proposition \ref{p:wedge} implies
  that each irreducible constituent $\sigma$ of $\bigwedge^*V_\bC$
  occurs in exactly $\dim\sigma$ different such simple $\bH$-modules.
\end{enumerate}
\end{remark}

\begin{example}
 In the case of the Hecke algebra of type $A_{n-1}$, one can
  determine exactly which constituent of $\bigwedge^*V_\bC,$ where
  $V_\bC\cong \bC^{n-1}$, appears in each simple $\bH$-module with
  central character $\rho^\vee.$ Consider a composition of $n$, i.e., a $k$-tuple
  $(n_1,n_2,\dots,n_k)$ where $n_i>0$, $\sum n_i=n$ and let $\overline X
  (n_1,n_2,\dots,n_k)$ be the simple module at $\rho^\vee$
  corresponding to the subset of the simple roots
  $\C S(n_1,n_2,\dots,n_k):=\Pi\setminus\{\al_{n_1},\al_{n_1+n_2},\dots\}.$ Notice that the standard
  module $X(n_1,n_2,\dots,n_k)$ contains the hook $S_n$-representation
  $(k,1,1,\dots,1)=\bigwedge^{n-k} V_\bC$ with multiplicity
  $1$. Moreover, if $\C S(n_1,n_2,\dots,n_k)\subsetneq \C
  S(n_1',n_2',\dots,n'_{k'})$, then $k'<k$, and $(k,1,1,\dots,1)$ does
  not appear in the standard module
  $X(n'_1,n'_2,\dots,n'_{k'})$. Therefore, by (\ref{e-alt}), we see
  that $\bigwedge^{n-k} V_\bC$ appears with multiplicity $1$ in $\overline X
  (n_1,n_2,\dots,n_k)$.
\end{example}

\begin{corollary}\label{c:Dirac-coh-rho}
Suppose the root system is semisimple. For every simple $\bH$-module $\overline X$ with central character
$\rho^\vee$, $$H^D(\overline X)=S,$$
for an appropriate choice of spin module $S$. 
\end{corollary}

\begin{proof}
Let $\overline X_M$ be a simple $\bH$-module at $\rho^\vee$, $M\subset
\Pi.$ As remarked above, $\overline X_M$ is $\bullet$-unitary and
therefore $H^D(X)=\ker D.$ Since $S(\Omega_{\wti
  W})=\langle\rho^\vee,\rho^\vee\rangle=\langle\cc(\overline X_M),\cc(\overline X_M)\rangle,$
a known argument, e.g. \cite[Proposition 5.7]{BCT}, says that $\ker D$ is nonzero
if (in fact, if and only if) $S$ occurs in $\overline X_M\otimes
S$. But $\Hom_{\wti W}[\overline X_M\otimes S,S]=\Hom_W[\overline X_M,
S\otimes S^*],$ and 
\begin{equation}
\begin{aligned}
&S\otimes S^*={\bigwedge}^* V_\bC,\text{ when }\dim V_\bC\text{ is
  even},\\
&(S^++S^-)\otimes (S^++S^-)^*=2{\bigwedge}^* V_\bC,\text{ when }\dim V_\bC\text{ is
  odd}.\\
\end{aligned}
\end{equation}
Then Proposition \ref{p:wedge}, implies that $\dim\Hom_{\wti W}[\overline
X_M\otimes S,S]=1$ for some choice of $S$.

\end{proof}

\ifx\undefined\bysame
\newcommand{\bysame}{\leavevmode\hbox to3em{\hrulefill}\,}
\fi

\end{document}